\documentclass[10pt,a4paper,twoside,reqno]{amsart}
\usepackage[foot]{amsaddr}

\usepackage[a4paper,%bindingoffset=0.2in,
left=1in,right=1in,top=1.2in,bottom=1.2in,%
footskip=.25in]{geometry}

\usepackage{bm}
\usepackage{latexsym, amsmath, amstext, amssymb, amsfonts, amscd, bm, array, multirow, amsbsy, mathrsfs}
\usepackage{amsthm}
\usepackage{t1enc}
\usepackage[mathscr]{eucal}
\usepackage{indentfirst}
\usepackage{pb-diagram}
\usepackage{graphicx}
\usepackage{fancyhdr}
\usepackage{fancybox}
\usepackage{color}
\usepackage{tikz-cd}
\usetikzlibrary{arrows}
\usepackage[all]{xy}
\usepackage{hyperref}
\usepackage{tikz}
\usepackage{xparse}
\hypersetup{colorlinks=false,pdfborderstyle={/S/U/W 0}}
\usetikzlibrary{matrix}
\usepackage{upgreek}
\usepackage[utf8]{inputenc}
\usepackage[T1]{fontenc}
\usepackage{bbm}

\newcommand{\doi}[1]{\href{https://doi.org/#1}{\ttfamily #1}}
\usepackage[noadjust]{cite}

\makeatletter
\def\@defaultbiblabelstyle#1{[#1]}
\makeatother

\usepackage{url}
\theoremstyle{plain}
\newtheorem{thm}{Theorem}[section]

\newtheorem{prop}[thm]{Proposition}

\newtheorem{lemma}[thm]{Lemma}
\newtheorem{cor}[thm]{Corollary}

\theoremstyle{definition}
\newtheorem{definition}[thm]{Definition}

\theoremstyle{remark}
\newtheorem{remark}[thm]{Remark}

\newtheorem*{ack}{Acknowledgements}

\newcommand{\End}{\mathrm{End}}
\newcommand{\Aut}{\mathrm{Aut}}

\DeclareFontFamily{U}{rsf}{}
\DeclareFontShape{U}{rsf}{m}{n}{<5> <6> rsfs5 <7> <8> <9> rsfs7 <10-> rsfs10}{}
\DeclareMathAlphabet\Scr{U}{rsf}{m}{n}

\def\Z{\mathbb{Z}}

\def\R{\mathbb{R}}

\def\dd{\mathrm{d}}

\def\fra{\mathfrak{a}}

\newcommand{\be}{\begin{equation*}}
\newcommand{\ee}{\end{equation*}}
\newcommand{\ben}{\begin{equation}}
\newcommand{\een}{\end{equation}}
\newcommand{\beqa}{\begin{eqnarray*}}
\newcommand{\eeqa}{\end{eqnarray*}}
\newcommand{\beqan}{\begin{eqnarray}}
\newcommand{\eeqan}{\end{eqnarray}}

\newcommand{\Tr}{\mathrm{Tr}}

\def\cC{{\mathcal C}}

\def\cB{\Scr B}

\def\Spin{\mathrm{Spin}}

\def\Spin{\mathrm{Spin}}

\def\SO{\mathrm{SO}}

\def\cD{\mathcal{D}}

\def\cA{\mathcal{A}}
\def\cE{\mathcal{E}}

\def\cP{\mathcal{P}}

\def\cP{\mathcal{P}}

\def\cC{\mathcal{C}}

\def\G_2{\mathrm{G_2}}

\def\Aut{\mathrm{Aut}}

\def\Im{\mathrm{Im}}

\def\G{\mathrm{G}}

\def\R{\mathbb{R}}

\def\dd{\mathrm{d}}

\newcommand{\escal}[1]{\langle#1\rangle}

\def\Span{\mathrm{Span}}

\def\fX{\mathfrak{X}}

\usepackage{enumerate}

\newcolumntype{P}[1]{>{\centering\arraybackslash}p{#1}}

\usepackage[math]{anttor}
\usepackage{orcidlink}
\usepackage{enumitem}
\setlist[itemize]{leftmargin=*}

\begin{document}%Draft:\,\today

%\nocite{*}

\title[Parallel spinors for $\mathrm{G}_2^*$ and isotropic structures]{Parallel spinors for $\mathrm{G}_2^*$ and isotropic structures}

\author[Alejandro Gil-García]{Alejandro Gil-García$^{1,\dag}$ \orcidlink{0000-0002-9370-241X}}
\address{$^1$Beijing Institute of Mathematical Sciences and Applications, Beijing, China}
\email{$^{\dag}$alejandrogilgarcia@bimsa.cn}

\author[C. S. Shahbazi]{C. S. Shahbazi$^{2,\ddag}$ \orcidlink{0000-0003-1185-9569}}
\address{$^2$Departamento de Matem\'aticas, Universidad UNED - Madrid, Reino de Espa\~na}
\email{$^{\ddag}$cshahbazi@mat.uned.es} 
\address{$^2$Fakult\"at f\"ur Mathematik, Universit\"at Hamburg, Bundesrepublik Deutschland.}
\email{$^{\ddag}$carlos.shahbazi@uni-hamburg.de}
 
\begin{abstract}
We obtain a correspondence between irreducible real parallel spinors on pseudo-Riemannian manifolds $(M,g)$ of signature $(4,3)$ and solutions of an associated differential system for three-forms that satisfy a homogeneous algebraic equation of order two in the Kähler-Atiyah bundle of $(M,g)$. Applying this general framework, we obtain an intrinsic algebraic characterization of $\mathrm{G}_2^*$-structures as well as the first explicit description of isotropic irreducible spinors in signature $(4,3)$ that are parallel under a general connection on the spinor bundle. This description is given in terms of a coherent system of mutually orthogonal and isotropic one-forms and follows from the characterization of the stabilizer of an isotropic spinor as the stabilizer of a highly degenerate three-form that we construct explicitly. Using this result, we show that isotropic spinors parallel under a metric connection with torsion exist when the connection preserves the aforementioned coherent system. This allows us to construct a natural class of metrics of signature $(4,3)$ on $\mathbb{R}^7$ that admit spinors parallel under a metric connection with torsion.\bigskip

\medskip

\end{abstract} 

\maketitle

\setcounter{tocdepth}{1} %doesn't display subsections in TOC 

%\tableofcontents

% % % % % % % % % % % % % % % % % % % % % % % % % % % % % % % % % % % % % % 
% % % % % % % % % % % % % % % % % % % % % % % % % % % % % % % % % % % % % % 

\section*{Introduction}

% % % % % % % % % % % % % % % % % % % % % % % % % % % % % % % % % % % % % % 
% % % % % % % % % % % % % % % % % % % % % % % % % % % % % % % % % % % % % % 

This paper investigates real and irreducible parallel spinors on pseudo-Riemannian manifolds $(M,g)$ of signature $(4,3)$. By \emph{parallel spinor} we refer to spinors parallel under a general connection on a spinor bundle, where the latter is defined as a bundle of Clifford modules $S$ on $(M,g)$ equipped with an admissible bilinear pairing $\cB$ \cite{AC97,ACDVP05}. This general concept of parallel spinor subsumes many classes of special spinors that have been extensively studied in the literature on a case-by-case basis, such as parallel spinors, Killing spinors, Codazzi spinors, Cauchy spinors, skew-torsion parallel spinors, skew-Killing spinors or generalized Killing spinors, just to name a few \cite{BGM05,BM08,CS07,CD24,FM22,Galaev13,GHK21,Ikemakhen06,Ikemakhen07,Leitner03,MS14_Einstein,MS14_Graphs,MS14_Spheres,AM24}. In signatures $(p,q)$ congruent to zero, one or two modulo eight, hence covering the $(4,3)$ case, a general theory of parallel spinors can be developed by using the theory of \emph{spinorial exterior forms} \cite{CLS21,Sha24}. This theory is based on describing parallel spinors in terms of solutions to an equivalent differential system for an exterior form that is algebraically constrained in the K\"ahler-Atiyah bundle of the underlying pseudo-Riemannian manifold. This algebraically constrained exterior form is precisely the \emph{square} of the corresponding parallel spinor and for irreducible real spinors of real type was characterized in \cite{CLS21,Sha24} as an element of an explicit semi-algebraic body whose specific structure depends on the considered dimension and signature. In this note, we use this framework, which lies at the heart of the understanding of spinors as the \emph{square root of geometry}, whence contributing to its development in signature $(4,3)$. In this regard, we should note that even though the square of a spinor remains to be studied systematically, its use as a tool in spin geometry is classical and can be traced back at least to the seminal work of Dadok and Harvey on spinors and calibrations \cite{Harvey1990,DadokHarvey93} as well as to the foundational book of Michelsohn and Lawson \cite{Spin89}. We refer the reader to the more modern book \cite{Mei13} for more details about the use of exterior forms in spin geometry.

Pseudo-Riemannian geometry in signature $(4,3)$ made its first star appearance in the list of possible holonomy groups of non-symmetric irreducible pseudo-Riemannian manifolds \cite{Ber55,Bry87} through the Lie group $\mathrm{G}_2^*$, which is a non-compact real form of the complex exceptional Lie group $\G_2^{\mathbb{C}}$ and occurs as one of the possible holonomy groups of an irreducible and non-symmetric pseudo-Riemannian manifold of signature precisely $(4,3)$. Similarly to its more classical Riemannian counterpart, namely the compact Lie group $\mathrm{G}_2$, the Lie group $\mathrm{G}_2^*$ can be characterized as the stabilizer in $\mathrm{GL}(7,\R)$ of a special three-form $\phi$, whose orbit under the action of $\mathrm{GL}(7,\R)$ is open and therefore represents a \emph{stable} three-form in seven dimensions in the sense of \cite{Hitchin01}. As proven by Kath in \cite{Kat98,KathHabilitation}, the Lie group $\mathrm{G}_2^*$ also appears as the stabilizer in the identity component of the spin group $\Spin_0(4,3)$ of an irreducible real spinor of non-zero pseudo-norm in signature $(4,3)$, similarly to the description of $\G_2$ as the stabilizer of a non-zero irreducible real spinor in seven Euclidean dimensions, see for instance \cite{nearlyG2_97} and references therein. The pseudo-Riemannian $(4,3)$ case is, however, richer in that it contains another possibility that cannot occur in the Euclidean case, namely the case of an irreducible real spinor of zero pseudo-norm, to which we will refer as an \emph{isotropic spinor}. As proven by Kath in \cite{Kat98}, the stabilizer in $\Spin_0(4,3)$ of an isotropic spinor is isomorphic to a semi-direct product of $\mathrm{SL}(3,\mathbb{R})$ and a six-dimensional nilpotent Lie group, illustrating the diversity of irreducible real spinors in signature $(4,3)$. Ever since the seminal work of Kath, various references \cite{BK99,CLSS11,Kat13,Fre13,FL15,FK19,FFS22,LocalTypeI,LocalTypeII,LocalTypeIII} have studied special types of spinors and related holonomy reductions or special geometric structures in signature $(4,3)$, which however remains much less developed than its Riemannian counterpart. In this regard, we hope that this work can contribute towards establishing a general framework that can be applied to the study of the geometric problems associated with irreducible real spinors in signature $(4,3)$.

% % % % % % % % % % % % % % % % % % % % % % % % % % % % % % % % % % % % % % 

\subsection*{Main results and outline}

% % % % % % % % % % % % % % % % % % % % % % % % % % % % % % % % % % % % % % 
 
\begin{itemize}
\itemsep 0em
\item Section \ref{sec:algebraic_spinors} is devoted to computing the algebraic square of an irreducible real spinor, possibly of vanishing pseudo-norm, in signature $(4,3)$. This computation produces two new results. For non-isotropic spinors, it leads to an intrinsic characterization of $\G^{\ast}_2$-structures as solutions of a homogeneous algebraic equation of order two, see Theorem \ref{thm:stabilizerG2*}. This characterization is intrinsic since it does not involve any choice of privileged basis or \emph{standard form}. In the isotropic case, it leads to the description of the stabilizer of an isotropic irreducible spinor as the stabilizer of a special and highly degenerate three-form that we compute explicitly, see Theorem \ref{thm:null_spinor}. Key to these results is the algebraic characterization of spinorial exterior forms developed in \cite{Sha24}. 

\item In Section \ref{sec:differentialspinors} we study parallel spinors in signature $(4,3)$. First, we obtain a \emph{master formula} that describes a general irreducible real parallel spinor, perhaps of non-constant pseudo-norm, in terms of an equivalent differential system for its \emph{square} as a spinorial exterior form, see Proposition \ref{prop:intermediate_formula}. Elaborating on this result, in Theorem \ref{thm:master_formula} we present a characterization of the most general nowhere isotropic parallel spinor in terms of an algebraically constrained three-form satisfying a prescribed differential system, which we give explicitly. As a corollary, we obtain an \emph{obstruction} for a nowhere vanishing and isotropic spinor to be a parallel spinor relative to given data. We then consider the case of an isotropic parallel spinor, which we present in Theorem \ref{thm:null_diff_spinor} in terms of an equivalent differential system for an \emph{isotropic coherent system} of triples of one-forms, see Definition \ref{def:isotropicfamily}. As a particular case, we obtain a natural characterization of isotropic spinors parallel under a metric connection with torsion, see Corollary \ref{cor:spinor_with_torsion}, which we apply to construct a natural class of metrics on $\mathbb{R}^7$ that admit spinors parallel under a metric connection with torsion. Thanks to our formalism we can easily give explicit expressions for these metrics and the corresponding contorsion tensors, obtaining an obstruction for a general contorsion tensor to admit a connection preserving an isotropic spinor. To the best of our knowledge, these are the first results on isotropic irreducible spinors in signature $(4,3)$ since the seminal work of Kath \cite{Kat98}.
\end{itemize}

\noindent
As proposed originally by C.\ Lazaroiu, the theory we have developed here in signature $(4,3)$ can also be developed in seven and eight-dimensional Euclidean signatures, leading to new results \cite{LS24_Spin7} and to an alternative approach to the general theory developed in \cite{ACFH15}.

\begin{ack}
We are indebted to Vicente Cortés for explaining how to efficiently compute the stabilizer of a spinorial isotropic three-form, and more generally for his constant guidance and support. We also thank Diego Conti for his valuable comments on a preliminary version of this article.
\end{ack}

% % % % % % % % % % % % % % % % % % % % % % % % % % % % % % % % % % % % % % 
% % % % % % % % % % % % % % % % % % % % % % % % % % % % % % % % % % % % % % 

% % % % % % % % % % % % % % % % % % % % % % % % % % % % % % % % % % % % % % 
% % % % % % % % % % % % % % % % % % % % % % % % % % % % % % % % % % % % % % 

\section{Algebraic \texorpdfstring{$\mathrm{G}^{\ast}_2$}{G2*}-structures and spinorial exterior forms}
\label{sec:algebraic_spinors}

% % % % % % % % % % % % % % % % % % % % % % % % % % % % % % % % % % % % % % 
% % % % % % % % % % % % % % % % % % % % % % % % % % % % % % % % % % % % % % 

In this section, we obtain a novel and intrinsic algebraic characterization of the stabilizer of an irreducible real spinor in signature $(4,3)$. We will use this algebraic characterization in Section \ref{sec:differentialspinors} as the basis of our investigation on irreducible parallel spinors on a pseudo-Riemannian manifold of signature $(4,3)$. We begin with a description of the stabilizer of an irreducible real spinor in signature $(4,3)$ as presented in \cite{Kat98}.

% % % % % % % % % % % % % % % % % % % % % % % % % % % % % % % % % % % % % % 

\subsection{The stabilizer of an irreducible real spinor}\label{subsec:Kath}

% % % % % % % % % % % % % % % % % % % % % % % % % % % % % % % % % % % % % % 

Let $V$ be an oriented seven-dimensional real vector space equipped with a non-degenerate metric $h$ of signature $(4,3)$ and let $(V^*,h^*)$ be the quadratic space dual to $(V,h)$. Let $\mathrm{Cl}(4,3):=\mathrm{Cl}(V^*,h^*)$ be the Clifford algebra associated to $(V^*,h^*)$ defined in the convention:
\begin{equation*}
\theta^2=h^*(\theta,\theta)\, ,\quad\theta\in V^*\, .
\end{equation*} 

\noindent
Denote by $\pi$ the standard automorphism of $\mathrm{Cl}(4,3)$ and by $\tau$ the standard anti-automorphism. In signature $(4,3)$, the pseudo-Riemannian volume form $\nu\in\mathrm{Cl}(4,3)$ squares to the identity and belongs to the center of $\mathrm{Cl}(4,3)$. Therefore, $\mathrm{Cl}(4,3)$ is non-simple and splits as a direct sum of unital and associative algebras: 
\begin{equation}
\label{eq:splitting_Cl(4,3)}
\mathrm{Cl}(4,3)=\mathrm{Cl}_+(4,3)\oplus\mathrm{Cl}_-(4,3),
\end{equation} 

\noindent
where:
\begin{equation*}
\mathrm{Cl}_l(4,3)=\{x\in\mathrm{Cl}(4,3)\mid\nu x=lx\}=\tfrac{1}{2}(1+l\nu)\mathrm{Cl}(4,3),\quad l\in\{-1,+1\}.
\end{equation*}

\noindent
The algebra $\mathrm{Cl}_l(4,3)$ is simple and isomorphic to $\mathrm{Mat}(8,\R)$. Therefore, $\mathrm{Cl}(4,3)$ admits two irreducible left Clifford modules of real dimension eight: $$\gamma_l:\mathrm{Cl}(4,3)\to\mathrm{End}(\Sigma),$$ that correspond to the projection of $\mathrm{Cl}(4,3)$ to the factor $\mathrm{Cl}_l(4,3)$ composed with an isomorphism of the latter to the algebra $\mathrm{End}(\Sigma)$ of endomorphisms of an eight-dimensional real vector space $\Sigma$. These two Clifford modules are distinguished by the value they take at the volume form $\nu\in\mathrm{Cl}(4,3)$, namely: 
\begin{equation*}
\gamma_l(\nu)=l\,\mathrm{Id}\in\mathrm{End}(\Sigma).
\end{equation*} 

\noindent
It is shown in \cite[Proposition 2.1]{Sha24} that we can equip $\Sigma$ with a symmetric bilinear pairing $\Scr B$ that satisfies:
\begin{equation*}
\Scr B(\gamma_l(x)\varepsilon_1,\varepsilon_2)=\Scr B(\varepsilon_1,\gamma_l((\pi\circ\tau)(x))\varepsilon_2),\quad x\in\mathrm{Cl}(4,3),\quad \varepsilon_1,\varepsilon_2\in\Sigma,\quad l\in\{-1,+1\}.
\end{equation*}

\noindent
Such pairing is called an \emph{admissible pairing} \cite{AC97,ACDVP05}. The admissible pairing $\Scr B$ is invariant under the action of the identity component $\Spin_0(4,3)\subset\Spin(4,3)$ of the spin group $\Spin(4,3)$.

\begin{definition}
The triple $(\Sigma,\gamma_l,\Scr B)$ is an \emph{irreducible paired Clifford module} if $(\Sigma,\gamma_l)$ is an irreducible Clifford module and $\Scr B$ is an admissible pairing on $(\Sigma,\gamma_l)$.
\end{definition}

\noindent
It was shown in \cite[Corollary 2.1]{Kat98} that if $\varepsilon\in\Sigma$ is a spinor with $\Scr B(\varepsilon,\varepsilon)\neq 0$ then its stabilizer:
\begin{equation*}
\mathrm{Stab}(\varepsilon)=\{x\in\Spin_0(4,3)\mid\gamma_l(x)\varepsilon=\varepsilon\}
\end{equation*}

\noindent
in $\Spin_0(4,3)$ is the Lie group $\mathrm{G}_2^*$. This is the connected non-compact dual of the Lie group $\mathrm{G}_2$ with trivial center and fundamental group $\Z_2$. It was also shown in \cite[Corollary 2.1]{Kat98} that the Lie algebra of the stabilizer of a non-zero spinor of zero norm is a semi-direct product of a six-dimensional nilpotent Lie algebra and $\mathfrak{sl}(3,\R)$ (see also Proposition \ref{prop:stabilizer_isotropic}). We call a spinor of zero norm an \emph{isotropic spinor}, and we will refer to the reduction defined by an isotropic spinor as an \emph{isotropic spinorial structure}.

\begin{remark}\label{remark:decomposition_Sigma}
Given a spinor $\varepsilon\in\Sigma$ of non-zero norm, we can decompose $\Sigma = \escal{\varepsilon} \oplus \escal{\varepsilon}^\perp$ via:
\begin{equation*}
\Sigma\ni\eta=\tfrac{\Scr B(\eta,\varepsilon)}{\Scr B(\varepsilon,\varepsilon)}\varepsilon+\left(\eta-\tfrac{\Scr B(\eta,\varepsilon)}{\Scr B(\varepsilon,\varepsilon)}\varepsilon\right)\, .  
\end{equation*}

\noindent
It follows that the space $\escal{\varepsilon}^\perp$ is isomorphic to $\{v\cdot \varepsilon\mid v\in V\}$ by \cite[Corollary 2.2]{Kat98}.
\end{remark}

\noindent
Consider now the universal covering $\lambda:\Spin_0(4,3)\to\mathrm{SO}_0(4,3)$ and for ease of notation denote again by $\mathrm{G}_2^*$ the image of $\mathrm{G}_2^*\subset\Spin(4,3)$ in $\mathrm{SO}(4,3)$, namely $\lambda(\mathrm{G}_2^*) \subset \mathrm{SO}(4,3)$. As a subgroup of $\SO(4,3)$ and consequently of $\mathrm{GL}(7,\mathbb{R})$, the Lie group $\G^{\ast}_2$ can be described as the stabilizer of a special three-form $\phi\in\Lambda^3V^*$ given by:
\begin{equation}
\label{eq:canonicalform}
\phi=-e^{127}-e^{135}+e^{146}+e^{236}+e^{245}-e^{347}+e^{567}
\end{equation}

\noindent
in an adequate basis $(e^1, \hdots ,  e^7)$ of $V^{\ast}$. In particular, $\escal{\phi,\phi} = -7$. Then:
\begin{equation*}
\mathrm{G}_2^*=\{A\in \Aut(V)\mid A^*\phi=\phi\}.
\end{equation*} 

\noindent
For further reference, and following \cite[Proposition 2.4]{Kat98}, we describe below the decomposition of $\Lambda^k V^*$ into irreducible components of the $\mathrm{G}_2^*$-action: 

\begin{enumerate}
    \itemsep 0em
    \item $\Lambda^1V^*=\Lambda^1_7$ is irreducible.
    \item $\Lambda^2V^*=\Lambda^2_7\oplus\Lambda^2_{14}$, where: \begin{align*}
        \Lambda^2_{7}&=\{\omega\in\Lambda^2V^*\mid*(\phi\wedge\omega)=2\omega\}=\{\iota_v\phi\mid v\in V\},\\
        \Lambda^2_{14}&=\{\omega\in\Lambda^2V^*\mid*(\phi\wedge\omega)=-\omega\}\cong\mathfrak{g}_2^*.
    \end{align*}
    \item $\Lambda^3V^*=\Lambda^3_1\oplus\Lambda^3_7\oplus\Lambda^3_{27}$, where: \begin{align*}
        \Lambda^3_1&=\{t\phi\mid t\in\R\}=\escal{\phi},\\
        \Lambda^3_7&=\{*(\phi\wedge\theta)\mid\theta\in\Lambda^1V^*\},\\
        \Lambda^3_{27}&=\{\rho\in\Lambda^3V^*\mid\rho\wedge\phi=0,\,\rho\wedge*\phi=0\}.
    \end{align*}
\end{enumerate}

\noindent
This decomposition is completely analogous to the one occurring for $\G_2$-structures in seven Euclidean dimensions. As we will see later, this type of decomposition does no longer hold for isotropic spinorial structures. 

% % % % % % % % % % % % % % % % % % % % % % % % % % % % % % % % % % % % % % 

\subsection{Spinorial exterior forms}\label{subsec:Shabazi}

% % % % % % % % % % % % % % % % % % % % % % % % % % % % % % % % % % % % % % 

In this subsection, we briefly revisit the formalism of spinorial exterior forms that describes under which algebraic conditions an exterior form is the square of an irreducible real spinor. We follow the exposition of \cite{Sha24} for the signature $(p,q)=(4,3)$. In order to construct the \emph{square} of a spinor as an exterior form, we use the natural identification: 
\begin{equation}
\label{eq:natural_identification}
\Psi\colon (\Lambda V^*,\diamond) \to \mathrm{Cl}(4,3) 
\end{equation} 

\noindent
between $\mathrm{Cl}(4,3)$ and the K\"ahler-Atiyah algebra of $(V^*,h^*)$, which we denote by $(\Lambda V^*,\diamond)$ \cite{Chev54,Chev55}. The map $\diamond:\Lambda V^*\times\Lambda V^*\to\Lambda V^*$ denotes the \emph{geometric product} determined by $h$. This is given by the linear and associative extension of the following expression: 
\begin{equation*}
\theta\diamond\alpha=\theta\wedge\alpha+\iota_{\theta^\sharp}\alpha\, ,\quad\theta\in V^*\, , \quad \alpha\in\Lambda V^*.
\end{equation*} 

\noindent
In order to do computations with the geometric product it is convenient to introduce the \emph{generalized products} of $(V^*,h^*)$. These are the bilinear operators $\triangle_k:\Lambda^aV^*\times\Lambda^bV^*\to\Lambda^{a+b-2k}V^*$, where $k=0,\ldots,7$, defined through the expansion:
\begin{equation*}
\alpha\diamond\beta=\sum_{k=0}^7(-1)^{\binom{k+1}{2}+ak}\alpha\triangle_k\beta,\quad\alpha\in\Lambda^a V^*\, , \quad \beta\in\Lambda V^*.
\end{equation*}
 
\noindent
Choosing a basis $\{e_1,\ldots,e_7\}$ of $V$ we can express the generalized products as $$\alpha\triangle_k\beta=\tfrac{1}{k!}h^{i_1j_1}\cdots h^{i_kj_k}(\iota_{e_{i_1}}\ldots\iota_{e_{i_k}}\alpha)\wedge(\iota_{e_{j_1}}\ldots\iota_{e_{j_k}}\beta).$$

\noindent
We collect below some useful properties of the generalized products that we will use in later computations.

\begin{prop}[{\cite{LB13,LBC13,LBC16}}]
Let $\alpha\in\Lambda^aV^*$ and $\beta\in\Lambda^bV^*$. Then: 
\begin{itemize}
\itemsep 0em
\item $\alpha\triangle_k\beta=0$ if $k>a$ or $k>b$.
\item $\alpha\triangle_k\beta=(-1)^{(a-k)(b-k)}\beta\triangle_k\alpha$. In particular $\alpha\triangle_k\alpha=0$ if $a-k$ is odd.
\item $\alpha\triangle_0\beta=\alpha\wedge\beta$ and if $b=a$ then $\alpha\triangle_a\beta=\escal{\alpha,\beta}$.
\end{itemize}
\end{prop}

\noindent
Now we transport the splitting \eqref{eq:splitting_Cl(4,3)} of $\mathrm{Cl}(4,3)$ to $(\Lambda V^*,\diamond)$ through the natural identification \eqref{eq:natural_identification}. This gives: 
\begin{equation*}
 (\Lambda V^*,\diamond)=(\Lambda_+V^*,\diamond)\oplus(\Lambda_-V^*,\diamond),
\end{equation*}

\noindent
where:
\begin{equation*}
\Lambda_lV^*=\{\alpha\in\Lambda V^*\mid\nu\diamond\alpha=l\alpha\}=\{\alpha\in\Lambda V^*\mid*\tau(\alpha)=l\alpha\},\quad l\in\{-1,+1\}.
\end{equation*}
 
\noindent
Here we have used the identity:
\begin{equation}
\label{eq:product_volume}
\alpha\diamond\nu=\nu\diamond\alpha=*\tau(\alpha).
\end{equation}    

\noindent
By composing the natural identification $\Psi\colon (\Lambda V^*,\diamond) \to \mathrm{Cl}(4,3)$ with the irreducible representation $\gamma_l:\mathrm{Cl}(4,3)\to\mathrm{End}(\Sigma)$ we obtain a surjective morphism of unital associative real algebras that we denote by:
\begin{equation*}
\Psi_l:=\gamma_l\circ\Psi\colon (\Lambda V^*,\diamond)\to\mathrm{End}(\Sigma).
\end{equation*}

\noindent
Let $\cP_l:\Lambda V^*\to\Lambda_l V^*$ be the natural projection given explicitly by:
\begin{equation}
\label{eq:projection_l}
\cP_l(\alpha)=\tfrac{1}{2}(1+l\nu)\diamond\alpha=\tfrac{1}{2}(\alpha+l*\tau(\alpha)),
\end{equation} 

\noindent
where we have used the identity \eqref{eq:product_volume}, and consider the canonical linear inclusion $\iota_l:\Lambda_lV^*\hookrightarrow\Lambda V^*$, which is a right inverse to $\cP_l$. In particular:
\begin{equation*}
\Psi_l\circ\iota_l:(\Lambda_lV^*,\diamond)\to\mathrm{End}(\Sigma)
\end{equation*}

\noindent
is an isomorphism of unital and associative algebras. Following \cite{LBC13,LB13,LBC16}, we define: 
\begin{equation*}
\Lambda^<V^*:=\Lambda^0V^*\oplus\Lambda^1V^*\oplus\Lambda^2V^*\oplus\Lambda^3V^*\subset\Lambda V^*.
\end{equation*} 

\noindent
Note that $\Lambda V^*=\Lambda^<V^*\oplus*\Lambda^<V^*$. Then, restricting $\cP_l:\Lambda V^*\to\Lambda_lV^*$ to $\Lambda^<V^*$ we obtain an isomorphism of vector spaces that we can use to transport the algebra product in $(\Lambda_lV^*,\diamond)$ to $\Lambda^<V^*$. For every pair $\alpha,\beta\in\Lambda^<V^*$ we then have:
\begin{equation*}
\alpha\vee\beta=2\cP_<(\cP_l(\alpha\diamond\beta)),
\end{equation*}

\noindent
where $\cP_<:\Lambda V^*\to\Lambda^<V^*$ is the natural projection. By construction, $(\Lambda_lV^*,\diamond)$ and $(\Lambda^<V^*,\vee)$ are naturally isomorphic as unital and associative real algebras. Given an irreducible paired Clifford module $(\Sigma,\gamma_l,\Scr B)$ and a sign $\mu\in\{-1,+1\}$, we introduce the following quadratic map: 
\begin{equation*}
    \cE^\mu:\Sigma\to\End(\Sigma),\quad\varepsilon\mapsto\mu\,\varepsilon\otimes\varepsilon^*,
\end{equation*}

\noindent
where $\varepsilon^*:=\Scr B(-,\varepsilon)\in\Sigma^*$. Using the isomorphism of unital and associative algebras:
\begin{equation*}
    \Psi_l^<:=\Psi_l\circ\iota_l\circ\cP_l|_{\Lambda^<V^*}:(\Lambda^<V^*,\vee)\to\End(\Sigma)
\end{equation*}

\noindent
we define the \emph{spinor square map} associated to $(\Sigma,\gamma_l,\Scr B)$ by:
\begin{equation*}
\cE^\mu_l:=(\Psi_l^<)^{-1}\circ\cE^\mu:\Sigma\to\Lambda^<V^*.
\end{equation*}

\noindent
All together we obtain the following diagram: $$\begin{tikzcd}
                                                                                                           &  & \Sigma \arrow[d, "\mathcal{E}^\mu"'] \arrow[rrd, "\mathcal{E}_l^\mu"]                                                                     &  &                                                                                                           \\
{\mathrm{Cl}(4,3)} \arrow[rr, "\gamma_l"]                                                                  &  & \mathrm{End}(\Sigma)                                                                                                                      &  & {(\Lambda^<V^*,\vee)} \arrow[lldd, "\mathcal{P}_l|_{\Lambda^<V^*}"', shift right] \arrow[ll, "\Psi_l^<"'] \\
                                                                                                           &  &                                                                                                                                           &  &                                                                                                           \\
{(\Lambda V^*,\diamond)} \arrow[uu, "\Psi"] \arrow[rr, "\mathcal{P}_l", shift left] \arrow[rruu, "\Psi_l"] &  & {(\Lambda_lV^*,\diamond)} \arrow[ll, "\iota_l", shift left] \arrow[rruu, "2\mathcal{P}_<"', shift right] \arrow[uu, "\Psi_l\circ\iota_l"] &  &                                                                                                          
\end{tikzcd}$$

\noindent
Given a spinor $\varepsilon\in\Sigma$, the exterior forms $\cE_l^+(\varepsilon)$ and $\cE_l^-(\varepsilon)$ are, respectively, the positive and negative \emph{squares} of $\varepsilon$ relative to the admissible pairing $\Scr B$. We will say that an exterior form $\alpha\in\Lambda^<V^*$ is the \emph{square} of a spinor $\varepsilon\in\Sigma$ if $\alpha=\cE_l^\mu(\varepsilon)$ for some $\mu\in\{-1,+1\}$, and we will refer generically to elements in the image of $\cE^\mu_l$ as \emph{spinorial exterior forms}. One of the key properties of the spinor square map, that we will use extensively, is that it is equivariant in the appropriate sense.

\begin{prop} 
\label{prop:equivariance}
The spinor square map $\cE^\mu_l:=(\Psi_l^<)^{-1}\circ\cE^\mu:\Sigma\to\Lambda^<V^*$ is equivariant with respect to the double covering morphism $\lambda\colon \Spin_0(4,3) \to \SO_0(4,3)$, that is:
\begin{equation*}
\cE^\mu_l(\gamma_l(x)\varepsilon) = \lambda(x)(\cE^\mu_l(\varepsilon))\quad \forall\,\, x\in \Spin_0(4,3)\quad \forall\,\, \varepsilon \in \Sigma,
\end{equation*}

\noindent
where the right-hand side denotes the natural action of $\lambda(x)\in\SO_0(4,3)$ on $\Lambda V^{\ast}$.
\end{prop}

\noindent
Hence, roughly speaking the stabilizers of a spinor and its associated spinorial exterior form can differ at most by a $\mathbb{Z}_2$ quotient. The following result, which follows as a particular case of \cite[Theorem 2.8]{Sha24}, see also \cite{LBC16,LBC13}, gives the algebraic characterization of spinors in terms of exterior forms in signature $(4,3)$ that is at the basis of our results.

\begin{cor}
\label{cor:characterization_p-q=1}
Let $(\Sigma,\gamma_l,\Scr B)$ be an irreducible paired Clifford module. The following statements are equivalent for an exterior form $\alpha\in\Lambda^<V^*$: 
\begin{enumerate}
\itemsep 0em
\item $\alpha$ is the square of a spinor $\varepsilon\in\Sigma$, namely $\alpha=\cE_l^\mu(\varepsilon)$ for some $\mu\in\{-1,+1\}$.
\item $\alpha$ satisfies the following relations: $$\alpha\vee\alpha=8\alpha^{(0)}\alpha,\quad(\pi \circ\tau)(\alpha)= \alpha,\quad\alpha\vee\beta\vee\alpha=8(\alpha\vee\beta)^{(0)}\alpha,$$ for a fixed exterior form $\beta\in\Lambda^<V^*$ such that $(\alpha\vee\beta)^{(0)}\neq0$.
\item The following relations hold: $$(\pi\circ\tau)(\alpha)= \alpha,\quad\alpha\vee\beta\vee\alpha=8(\alpha\vee\beta)^{(0)}\alpha,$$ for every exterior form $\beta\in\Lambda^<V^*$.
\end{enumerate}
\end{cor}

\noindent
As it follows from the second point above, if $\varepsilon$ is of non-zero pseudo-norm then it suffices to take $\beta = 1$. This is in general not enough if $\varepsilon$ is isotropic and illustrates the difficulty of this case in relation to the generic case of non-zero pseudo-norm.

% % % % % % % % % % % % % % % % % % % % % % % % % % % % % % % % % % % % % % 

\subsection{The non-degenerate algebraic characterization}
\label{subsec:algebraic_char}

% % % % % % % % % % % % % % % % % % % % % % % % % % % % % % % % % % % % % % 

In this subsection, we use Corollary \ref{cor:characterization_p-q=1} to give an intrinsic algebraic characterization of $\mathrm{G}_2^*$-structures in terms of exterior forms. An element $\alpha\in\Lambda^<V^*$ can be expressed as $\alpha=\alpha^{(0)}+\alpha^{(1)}+\alpha^{(2)}+\alpha^{(3)}$, where $\alpha^{(j)}\in\Lambda^jV^*$ for $j=0,1,2,3$. Recall that $$\pi(\alpha)=(-1)^a\alpha,\quad\tau(\alpha)=(-1)^{\binom{a}{2}}\alpha,\quad\alpha\in\Lambda^aV^*.$$

\noindent
The condition $(\pi\circ\tau)(\alpha)=\alpha$ is linear and simply implies that $\alpha^{(1)}=\alpha^{(2)}=0$. Set $c:=\alpha^{(0)}\in\R=\Lambda^0V^*$ and $\phi:=\alpha^{(3)}\in\Lambda^3V^*$. Now we compute $\alpha\vee\alpha$. First, we need to obtain $\alpha\diamond\alpha$: 
\begin{equation*}
\alpha\diamond\alpha=(c+\phi)\diamond(c+\phi)=c^2+2c\phi+\phi\diamond\phi=c^2+2c\phi+\phi\triangle_1\phi-\escal{\phi,\phi}.
\end{equation*}

\noindent
Using now \eqref{eq:projection_l} together with the previous expression, we obtain: 
\begin{equation*}
\cP_l(\alpha\diamond\alpha)=\tfrac{1}{2}\big(c^2+2c\phi+\phi\triangle_1\phi -\escal{\phi,\phi}+l(*c^2+2c*\tau(\phi)+*\tau(\phi\triangle_1\phi)-*\escal{\phi,\phi})\big),
\end{equation*}

\noindent
from which we conclude:
\begin{equation*}
    \alpha\vee\alpha=2\cP_<(\cP_l(\alpha\diamond\alpha))=c^2+2c\phi-\escal{\phi,\phi}+l*(\phi\triangle_1\phi),
\end{equation*}

\noindent
where we have used that $\tau(\phi\triangle_1\phi)=\phi\triangle_1\phi\in\Lambda^4V^*$. Hence the condition $\alpha\vee\alpha=8\alpha^{(0)}\alpha$ implies:
\begin{equation*}
\escal{\phi,\phi}=-7c^2,\quad 6c\phi=l*(\phi\triangle_1\phi).
\end{equation*}

\noindent
Summarizing this discussion we obtain the following result.

\begin{prop}
\label{prop:algebraic_characterization}
An exterior form $\alpha\in\Lambda^<V^*$ is the square of a spinor $\varepsilon\in\Sigma$ only if $\alpha=c+\phi$, with $c\in\R$ and $\phi\in\Lambda^3V^*$, satisfying:
\begin{equation}
\label{eq:algebraic_characterization}
\escal{\phi,\phi}=-7c^2,\quad 6c\phi=l*(\phi\triangle_1\phi).
\end{equation}

\noindent
If in addition $\cB(\varepsilon,\varepsilon) \neq 0$, then $\escal{\phi,\phi} < 0$ and the previous condition is also sufficient.
\end{prop}

\begin{proof}
If $\cB(\varepsilon,\varepsilon) \neq 0$ then $c\neq0$ as explained below, and thus taking $\beta=1$ in Corollary \ref{cor:characterization_p-q=1} we obtain an equivalence since in this case $(\alpha\vee\beta)^{(0)}=\alpha^{(0)}=c\neq0$.
\end{proof}

\noindent
Hence, elaborating on Equation \eqref{eq:algebraic_characterization}, we conclude that $\alpha\in\Lambda^<V^*$ is the square of an irreducible real spinor in signature $(4,3)$ only if:
\begin{equation*}
\alpha = \kappa \sqrt{-\frac{\escal{\phi,\phi}}{7}}  + \phi,
\end{equation*}

\noindent
where $\kappa\in\{-1,+1\}$ is a sign and $\phi \in \Lambda^3 V^{\ast}$ satisfies:
\begin{equation*}
  6\,\kappa \sqrt{-\frac{\escal{\phi,\phi}}{7}}\,\phi=l*(\phi\triangle_1\phi).
\end{equation*}

\noindent
On the other hand, if $\alpha\in \Lambda^{<} V^{\ast}$ is the positive square of $\varepsilon\in \Sigma$, then from the general theory of spinorial exterior forms, see \cite{Sha24}, we obtain:
\begin{equation*}
\alpha = \frac{1}{8} \sum_{k=0}^{3}  \,\sum_{i_1 < \dots < i_k} \cB(\Psi_l^{<}(e^{i_k})^{-1} \cdots \Psi_l^{<}(e^{i_1})^{-1}\varepsilon,\varepsilon)\, e^{i_1}\wedge \hdots \wedge e^{i_k}   
\end{equation*}

\noindent
and thus:
\begin{equation*}
\cB(\varepsilon,\varepsilon)^2 = -\frac{64}{7} \escal{\phi,\phi}   
\end{equation*}

\noindent
and if $\cB(\varepsilon,\varepsilon) \neq 0$ then $\kappa$ is the sign of the pseudo-norm $\cB(\varepsilon,\varepsilon)$ of $\varepsilon$. Finally, Propositions \ref{prop:equivariance} and \ref{prop:algebraic_characterization} imply the following intrinsic characterization of $\G^{\ast}_2$-structures.

\begin{thm}
\label{thm:stabilizerG2*}
The stabilizer of a three-form $\phi\in \Lambda^3 V^{\ast}$ is $\G^{\ast}_2\subset \SO_0(4,3)$ if and only if $\escal{\phi,\phi}< 0$ and:
\begin{equation}
\label{eq:intrinsicG2}
  6\,\kappa \sqrt{-\frac{\escal{\phi,\phi}}{7}}\,\phi=l*(\phi\triangle_1\phi).
\end{equation}

\noindent
If that is the case, then the metric induced by $\phi$ is homothetic to the underlying metric $h$. If in addition $\escal{\phi,\phi} = -7$ then the metric induced by $\phi$ is precisely $h$.
\end{thm}

\begin{remark}
This result should be compared with \cite[Theorem 7.3]{KathHabilitation}, where an alternative algebraic characterization of $\G^{\ast}_2$ structures is obtained using different methods.
\end{remark}

\begin{proof}
Let $\phi\in \Lambda^3 V^{\ast}$ be a three-form satisfying $\escal{\phi,\phi}< 0$ as well as Equation \eqref{eq:intrinsicG2}. Then:
\begin{equation*}
\alpha = \kappa \sqrt{-\frac{\escal{\phi,\phi}}{7}}  + \phi
\end{equation*}

\noindent
is the positive square of an irreducible real spinor $\varepsilon \in \Sigma$, that is, $\cE^{+}_l(\varepsilon) = \alpha$. By \cite[Corollary 2.1]{Kat98}, $\phi$ is stabilized by $\G^{\ast}_2 \subset \Spin_0(4,3)$ and hence by Proposition \ref{prop:equivariance} it follows that $\phi$ is stabilized by $\lambda(\G^{\ast}_2)\subset \SO_0(4,3)$. Since $\phi$ is not necessarily of norm $-7$, it is in general a multiple of the canonical three-form $\eqref{eq:canonicalform}$ in the given basis, hence $\phi$ induces a metric that is homothetic to the metric $h$ induced by the canonical three-form \eqref{eq:canonicalform}. For the converse, assume that $\phi$ is stabilized by $\lambda(\G^{\ast}_2)\subset \SO_0(4,3)$. Then, we can lift $\lambda(\G^{\ast}_2)\subset \SO_0(4,3)$ via the double cover morphism $\lambda\colon \Spin_0(4,3) \to \SO_0(4,3)$ to a subgroup $\G^{\ast}_2 \subset \Spin_0(4,3)$. By \cite{Kat98} this subgroup can be realized as the stabilizer of an element $\varepsilon\in \Sigma$, which in turn defines a three-form as the degree three $\cE^{\mu}_l(\varepsilon)^{(3)}$ component of its square $\cE^{\mu}_l(\varepsilon) \in \Lambda^{<} V^{\ast}$. By equivariance of the spinor square map, it follows that $\cE^{\mu}_l(\varepsilon)$ and $\phi$ are preserved by the same subgroup $\lambda(\G^{\ast}_2)\subset \SO_0(4,3)$, and thus they must be proportional. Therefore, rescaling $\varepsilon$ if necessary and choosing adequately the sign $\mu \in \left\{-1,1\right\}$, we obtain that $\phi = \cE^{\mu}_l(\varepsilon)^{(3)}$ and therefore by Proposition \ref{prop:algebraic_characterization} we conclude.
\end{proof}

\noindent
As a corollary to the previous theorem and discussion, we obtain the following result for \emph{normalized} non-isotropic spinors.

\begin{cor}
An exterior form $\alpha\in \Lambda^{<} V^{\ast}$ is the square of a pseudo-unit spinor $\varepsilon\in\Sigma$ if and only if:
\begin{equation*}
\alpha = \frac{\kappa}{8} + \phi,
\end{equation*}

\noindent
where $\phi\in \Lambda^3 V^{\ast}$ satisfies:
\begin{equation*}
  6\,\kappa\,  \phi=8\,l*(\phi\triangle_1\phi).
\end{equation*}
\end{cor}

\noindent
We end this section with a technical lemma necessary for later computations.

\begin{lemma}
\label{lemma:aux-triangle1}
Let $c + \phi\in\Lambda^{<}V^*$ satisfying Equation \eqref{eq:algebraic_characterization} with $\escal{\phi,\phi} \neq 0$. Then: 
\begin{enumerate}
\itemsep 0em
\item $(\iota_v\phi)\triangle_1\phi=3lc*(\phi\wedge v^\flat)\in\Lambda^3_7$ for $v\in V$.
\item $\omega\triangle_1\phi=0$ for $\omega\in\Lambda^2_{14}\cong\mathfrak{g}_2^*$.
\item $*\big((*(\theta\wedge\phi))\triangle_1\phi\big)=3lc*(\theta\wedge\phi)\in\Lambda^3_7$ for $\theta\in V^*$.
\item $*(\rho\triangle_1\phi)=-3lc\rho$ for $\rho\in\Lambda^3_{27}$.
\end{enumerate}
\end{lemma}

\begin{proof}
    (1) We have $6c\phi=l*(\phi\triangle_1\phi)$, which is equivalent to $\phi\triangle_1\phi=-6lc*\phi$. Hence $$(\iota_v\phi)\triangle_1\phi=\phi(v,e_i)\wedge\phi(e_i)=-\tfrac{1}{2}\iota_v(\phi(e_i)\wedge\phi(e_i))=-\tfrac{1}{2}\iota_v(\phi\triangle_1\phi)=-\tfrac{1}{2}\iota_v(-6lc*\phi)=3lc*(\phi\wedge v^\flat).$$

    \noindent
    (2) Recall that $\Omega^*\phi=\phi$ for $\Omega\in\mathrm{G}_2^*$ by definition. Thus $\omega\cdot\phi=0$ for $\omega\in\mathfrak{g}_2^*$, where $\omega\cdot\phi\in\Lambda^3V^*$ is defined by $$(\omega\cdot\phi)(u,v,w)=\phi(\omega u,v,w)+\phi(u,\omega v,w)+\phi(u,v,\omega w),$$ which is expressed in terms of a basis by $$\omega\cdot\phi=\tfrac{1}{3!}(\omega\cdot\phi)_{ijk}e^{ijk}=\tfrac{1}{3!}(\omega^l_i\phi_{ljk}+\omega^l_j\phi_{ilk}+\omega^l_k\phi_{ijl})e^{ijk}=\tfrac{1}{2}\omega^l_i\phi_{ljk}e^{ijk}.$$

    \noindent
    Therefore $$\omega\triangle_1\phi=h^{ml}\iota_{e_m}\omega\wedge\iota_{e_l}\phi=\tfrac{1}{2}h^{ml}\omega_{mi}\phi_{ljk}e^{ijk}=\tfrac{1}{2}\omega^l_i\phi_{ljk}e^{ijk}=\omega\cdot\phi=0.$$
    
    \noindent
    (3) We compute \begin{align*}
        *\big((*(\theta\wedge\phi))\triangle_1\phi\big)&=*\big(\iota_{e_i}(*(\theta\wedge\phi))\wedge\phi(e_i)\big)=*\big(*(\theta\wedge\phi\wedge e^i)\wedge e^j\wedge e^k\big)\phi_{ijk}\\
        &=\iota_{e_j}\iota_{e_k}*^2(\theta\wedge\phi\wedge e^i)\phi_{ijk}=-\iota_{e_j}(\theta(e_k)\phi\wedge e^i-\theta\wedge\phi(e_k)\wedge e^i)\phi_{ijk}\\
        &=-(\theta(e_k)\phi(e_j)\wedge e^i-\theta(e_j)\phi(e_k)\wedge e^i+\theta\wedge\phi(e_k,e_j)\wedge e^i)\phi_{ijk}\\
        &=-2\phi_{ijk}\theta(e_k)\phi(e_j)\wedge e^i-\phi_{ijk}\theta\wedge\phi(e_k,e_j)\wedge e^i.
    \end{align*}

    \noindent
    Note that $\phi_{ijk}\theta\wedge\phi(e_k,e_j)\wedge e^i=\theta\wedge(\phi_{ijk}\phi_{kjl}e^{li})=\theta\wedge(\phi\triangle_2\phi)=0$. We also have that \begin{align*}
        \iota_{\theta^\sharp}(\phi\triangle_1\phi)&=\iota_{\theta^\sharp}(\phi(e_i)\wedge\phi(e_i))=2\phi(e_i,\theta^\sharp)\wedge\phi(e_i)=-2(\iota_{\theta^\sharp}\phi(e_i))\wedge\phi(e_i)\\
        &=-\phi_{iab}\phi_{icd}\iota_{\theta^\sharp}(e^{ab})\wedge e^{cd}=-2\phi_{iab}\phi_{icd}\theta(e_a)e^{bcd}=-4\phi_{iab}\theta(e_a)\phi(e_i)\wedge e^b.
    \end{align*}

    \noindent Hence $$*\big((*(\theta\wedge\phi))\triangle_1\phi\big)=\tfrac{1}{2}\iota_{\theta^\sharp}(\phi\triangle_1\phi)=\tfrac{1}{2}\iota_{\theta^\sharp}(-6lc*\phi)=-3lc*(\phi\wedge\theta).$$
    
    \noindent
    (4) Here we use the description of $\Lambda^3_{27}$ given in \cite[Equation 16]{Kar09}, that is $$\Lambda^3_{27}=\{A_{ij}h^{jk}e^i\wedge\phi(e_k)\mid A_{ij}=A_{ji},\,\Tr_h(A)=h^{ij}A_{ij}=0\}.$$

    \noindent
    Although this description is originally given for $\mathrm{G}_2$-structures it also holds in the $\mathrm{G}_2^*$ case. Let $\rho\in\Lambda^3_{27}$ given by $\rho=A_i^ke^i\wedge\phi(e_k)$, then \begin{align*}
        \rho\triangle_1\phi&=\iota_{e_a}(A_i^ke^i\wedge\phi(e_k))\wedge\phi(e_a)=A_i^k(h^i_a\phi(e_k)-e^i\wedge\phi(e_k,e_a))\wedge\phi(e_a)\\
        &=A_i^k\phi(e_k)\wedge\phi(e_i)-A_i^ke^i\wedge\phi(e_k,e_a)\wedge\phi(e_a).
    \end{align*}

    \noindent
    For the first term note that the metric induced by the three-form $\phi$ is a multiple of the original metric $h$ and then the quantity $\phi(e_k)\wedge\phi(e_i)$ is just a multiple of $h_k^i*\phi$. Hence the first term is a multiple of $\Tr_h(A)*\phi$, thus is zero. For the second term we have $$A_i^ke^i\wedge\phi(e_k,e_a)\wedge\phi(e_a)=-\tfrac{1}{2}A_i^ke^i\wedge\iota_{e_k}(\phi\triangle_1\phi)=3lcA_i^ke^i\wedge\iota_{e_k}(*\phi)=3lcA_i^ke^i\wedge*(\phi\wedge e^k).$$

    \noindent
    Applying the Hodge star operator to this expression we obtain \begin{align*}
        3lcA_i^k*(e^i\wedge*(\phi\wedge e^k))&=-3lcA_i^k\iota_{e_i}*^2(\phi\wedge e^k)=3lcA_i^k\phi(e_i)\wedge e^k-3lcA_i^kh_i^k\phi=3lc\rho,
    \end{align*} since $A_i^kh_i^k=\Tr_h(A)=0$.
\end{proof}

% % % % % % % % % % % % % % % % % % % % % % % % % % % % % % % % % % % % % % 

\subsection{The isotropic algebraic characterization}
\label{subsec:algebraic_charisotropic}

% % % % % % % % % % % % % % % % % % % % % % % % % % % % % % % % % % % % % % 
 
Proposition \ref{prop:algebraic_characterization} does not give sufficient conditions for an exterior form to be the square of an isotropic spinor, it only provides a necessary condition. Therefore the isotropic case, namely the case  $\cB(\varepsilon,\varepsilon) = 0$, needs to be considered separately. By Corollary \ref{cor:characterization_p-q=1}, we need to find an element $\beta\in \Lambda^{<} V^{\ast}$ such that $(\alpha\vee\beta)^{(0)}\neq 0$. In order to do this, we first note that it is enough to take $\beta\in\Lambda^3V^*$.

\begin{lemma}\label{lemma:(alpha-beta)0}
    Let $\beta\in\Lambda^<V^*$ and $\phi\in\Lambda^3V^*$. Then $(\phi\vee\beta)^{(0)}=-\escal{\phi,\beta^{(3)}}$.
\end{lemma}

\begin{proof}
    Let $\beta=\beta^{(0)}+\beta^{(1)}+\beta^{(2)}+\beta^{(3)}$. We first have $\phi\diamond\beta=\phi\beta^{(0)}+\phi\diamond\beta^{(1)}+\phi\diamond\beta^{(2)}+\phi\diamond\beta^{(3)}$ with \begin{align*}
        \phi\diamond\beta^{(1)}&=\phi\wedge\beta^{(1)}+\phi\triangle_1\beta^{(1)},\\
        \phi\diamond\beta^{(2)}&=\phi\wedge\beta^{(2)}+\phi\triangle_1\beta^{(2)}-\phi\triangle_2\beta^{(2)},\\
        \phi\diamond\beta^{(3)}&=\phi\wedge\beta^{(3)}+\phi\triangle_1\beta^{(3)}-\phi\triangle_2\beta^{(3)}-\escal{\phi,\beta^{(3)}}.
    \end{align*} Recall that $\cP_l(\alpha)=\frac{1}{2}(\alpha+l*\tau(\alpha))$ for any exterior form $\alpha$ and $\cP_<:\Lambda V^*\to\Lambda^<V^*$ is the projection. Hence \begin{align*}
        \phi\vee\beta^{(1)}&=2\cP_<(\cP_l(\phi\diamond\beta^{(1)}))=\phi\triangle_1\beta^{(1)}+l*\tau(\phi\wedge\beta^{(1)}),\\
        \phi\vee\beta^{(2)}&=2\cP_<(\cP_l(\phi\diamond\beta^{(2)}))=\phi\triangle_1\beta^{(2)}-\phi\triangle_2\beta^{(2)}+l*\tau(\phi\wedge\beta^{(2)}),\\
        \phi\vee\beta^{(3)}&=2\cP_<(\cP_l(\phi\diamond\beta^{(3)}))=-\phi\triangle_2\beta^{(3)}-\escal{\phi,\beta^{(3)}}+l*\tau(\phi\wedge\beta^{(3)})+l*\tau(\phi\triangle_1\beta^{(3)}).
    \end{align*} Therefore $(\phi\vee\beta)^{(0)}=-\escal{\phi,\beta^{(3)}}$.
\end{proof}

\noindent
To proceed further we need to introduce a basis adapted to the geometry of the problem. Since $h$ is of signature $(4,3)$ we can always choose a unit-norm element $n\in V^{\ast}$. This element induces a decomposition:
\begin{equation*}
V^{\ast} = V^{\ast}_n \oplus \langle  n\rangle,
\end{equation*}

\noindent
where $V^{\ast}_n\subset V^{\ast}$ is the orthogonal complement of $n$ in $V^{\ast}$. Hence, with the metric induced by $h$, it follows that $V^{\ast}_n$ becomes a six-dimensional vector space of split signature $(3,3)$. We choose a basis $(\vartheta_+,\theta_+,\eta_+,\vartheta_-,\theta_-,\eta_-)$ on $V^{\ast}_n$ given by isotropic one-forms that are conjugate by pairs, that is, they are orthogonal except for:
\begin{equation*}
\escal{\vartheta_+,\vartheta_-}=\escal{\theta_+,\theta_-}=\escal{\eta_+,\eta_-}=1.
\end{equation*}

\noindent
We refer to such $(\vartheta_+,\theta_+,\eta_+,\vartheta_-,\theta_-,\eta_-,n)$ as an \emph{isotropic} basis for $(V^{\ast},h^{\ast})$.

\begin{thm}\label{thm:null_spinor}
An exterior form $\alpha\in\Lambda^<V^*$ is the square of an isotropic spinor if and only if $\alpha = \vartheta \wedge \theta \wedge \eta \in\Lambda^3V^*$ for $\vartheta,\theta,\eta\in V^*$ isotropic and orthogonal.
\end{thm}

\begin{proof}
Suppose that $\alpha=\vartheta\wedge\theta\wedge\eta$ for isotropic and orthogonal one-forms $\vartheta,\theta,\eta\in V^*$. Set $\vartheta_+ := \vartheta$, $\theta_+ := \theta$ and $\eta_+:= \eta$ and complete these elements into an isotropic basis $(\vartheta_+,\theta_+,\eta_+,\vartheta_-,\theta_-,\eta_-,n)$. Take $\beta =\vartheta_-\wedge\theta_-\wedge\eta_-$. We have that $\escal{\alpha,\beta}=1$, hence $(\alpha \vee \beta)^{(0)} \neq 0$, and a tedious but straightforward computation shows that $\alpha\vee\beta\vee\alpha=-8\alpha$. Thus, we conclude that $\alpha$ is the square of an isotropic spinor by Corollary \ref{cor:characterization_p-q=1}.  To prove the converse, we assume that there exists no isotropic basis  $(\vartheta_+,\theta_+,\eta_+,\vartheta_-,\theta_-,\eta_-,n)$ such that $\alpha=\vartheta_+\wedge\theta_+\wedge\eta_+$. Then, we can write:
\begin{equation*}
\alpha = \phi_0 + \phi^{\perp},
\end{equation*}

\noindent
where $\phi_0 = b\,\vartheta_+ \wedge\theta_+ \wedge \eta_+$ for a constant $b\in \mathbb{R}$ and $\phi^{\perp} (\vartheta^{\sharp}_+  , \theta^{\sharp}_+ , \eta^{\sharp}_+) = 0$. Solving now $\alpha\vee\beta\vee\alpha=-8\alpha$ with $\beta = \vartheta_- \wedge\theta_- \wedge \eta_-$ a long computation gives $\phi^{\perp} = 0$ and $b^2 =1$, which gives a contradiction. 
\end{proof}

\noindent
The theory of spinorial exterior forms that we have used allows, as exemplified by the previous theorem, to obtain results without relying on any knowledge about the representation theory of the stabilizer of the class of spinors under consideration. The price to pay however is dealing with long algebraic computations within a rich algebraic structure that contains all the necessary information. On the other hand, once the square of a spinor is computed, as in the previous theorem, the result can be typically used to simplify the study of the stabilizer of the corresponding spinor. We proceed to do so in the following for isotropic spinors. Let $u=\vartheta^\sharp,v=\theta^\sharp,w=\eta^\sharp\in V$ be the dual vectors of the one-forms in Theorem \ref{thm:null_spinor}. Hence, they are isotropic and orthogonal. We can decompose our vector space $V$ as $V=U\oplus W$, where $U=\escal{u,v,w}$ and we complete $(u,v,w)$ to an isotropic basis of $V$. That is, we choose a basis $W=\escal{u',v',w',n}$, where $u',v',w'$ are isotropic, orthogonal and conjugate to $(u,v,w)$, namely:
\begin{equation*}
\escal{u,u'}=\escal{v,v'}=\escal{w,w'}=1
\end{equation*}

\noindent
with $n\in V$ of unit-norm and orthogonal to $(u,v,w,u',v',w')$. In this basis the metric $h$ takes the form \begin{equation}\label{eq:matrix_null_basis}
    h=\begin{pmatrix}
    0&\mathbbm{1}_3&0\\
    \mathbbm{1}_3&0&0\\
    0&0&1
\end{pmatrix},
\end{equation}

\noindent
where $\mathbbm{1}_3$ denotes the $3\times3$ identity matrix. The fact that by Theorem \ref{thm:null_spinor} we explicitly know the square of an isotropic irreducible spinor together with the fact that the square spinor map is equivariant, allows us to easily compute the Lie algebra of an irreducible isotropic spinor, complementing the characterization given by Kath in \cite{Kat98}. We proceed to do this in the remainder of this subsection.

\begin{prop}
\label{prop:stabilizer_isotropic}
Let $\alpha=\vartheta\wedge\theta\wedge\eta$ be the square of an isotropic irreducible spinor $\varepsilon\in \Sigma$. Then the Lie algebra $\mathrm{Lie}(\mathrm{Stab}(\varepsilon))\subset\mathfrak{so}(4,3)$ of the stabilizer $\mathrm{Stab}(\varepsilon)$ of $\varepsilon$ in $\Spin_0(4,3)$ is given by:
\begin{equation*}
\left\{\begin{pmatrix}
A&B&v\\
0&-A^t&0\\
0&-v^t&0
\end{pmatrix}\mid A\in\mathfrak{sl}(3,\R),B\in\mathfrak{so}(3,\R),v\in\R^3\right\}\cong\mathfrak{sl}(3,\R)\ltimes(\Lambda^2\R^3\oplus\R^3),
\end{equation*}

\noindent
where $\Lambda^2\R^3\oplus\R^3$ denotes the unique six-dimensional two-step nilpotent Lie algebra with three-dimensional derived subalgebra.
\end{prop}

\begin{proof}
Let us consider 
\begin{equation*}
M=\begin{pmatrix}
        A&B&v\\
        C&D&u\\
        x^t&y^t&r
    \end{pmatrix}\in\mathfrak{gl}(7,\R),
\end{equation*}

\noindent
where $A,B,C,D\in\mathfrak{gl}(3,\R)$, $v,u,x,y\in\R^3$ and $r\in\R$. Since $M$ must belong to $\mathfrak{so}(V,h)\cong\mathfrak{so}(4,3)$, for $h$ as in \eqref{eq:matrix_null_basis}, we get that $B,C\in\mathfrak{so}(3,\R)$, $D=-A^t$, $x=-u$, $y=-v$ and $r=0$. Moreover, since $M$ must preserve $\alpha$, i.e.\ $M^*\alpha=\alpha$, then we get that $M(U)\subset U$ (where we have the decomposition $V=U\oplus W$), thus $C=0$, $u=0$ and $A\in\mathfrak{sl}(3,\R)$. The Lie algebra $\mathfrak{so}(3,\R)$ can be identified, as a vector space, with $\Lambda^2\R^3$. In fact, we have that $\Lambda^2\R^3$ is an abelian Lie algebra and the only non-zero brackets of $\Lambda^2\R^3\oplus\R^3$ are given by: 
\begin{equation*}
\left[\begin{pmatrix}
        0&0&v_1\\
        0&0&0\\
        0&-v_1^t&0
    \end{pmatrix},\begin{pmatrix}
        0&0&v_2\\
        0&0&0\\
        0&-v_2^t&0
    \end{pmatrix}\right]=\begin{pmatrix}
        0&-v_1\otimes v_2^t+v_2\otimes v_1^t&0\\
        0&0&0\\
        0&0&0
    \end{pmatrix}\in\Lambda^2\R^3,
\end{equation*}

\noindent
from where one can see that $\Lambda^2\R^3\oplus\R^3$ is two-step nilpotent and has three-dimensional derived subalgebra, which uniquely characterizes it.
\end{proof}

\noindent
This result was already obtained in \cite[Corollary 2.1]{Kat98} in a different way, and it is complemented here by explicitly obtaining that the six-dimensional nilpotent Lie algebra is $\Lambda^2\R^3\oplus\R^3$, which corresponds to $(0,0,0,12,13,23)$ in the notation of Salamon \cite{Salamon:ComplexStructures}.  

% % % % % % % % % % % % % % % % % % % % % % % % % % % % % % % % % % % % % % 
% % % % % % % % % % % % % % % % % % % % % % % % % % % % % % % % % % % % % % 

% % % % % % % % % % % % % % % % % % % % % % % % % % % % % % % % % % % % % % 
% % % % % % % % % % % % % % % % % % % % % % % % % % % % % % % % % % % % % % 

\section{Parallel spinors on \texorpdfstring{$\mathrm{G}^{\ast}_2$}{G2*}-manifolds}
\label{sec:differentialspinors}

% % % % % % % % % % % % % % % % % % % % % % % % % % % % % % % % % % % % % % 
% % % % % % % % % % % % % % % % % % % % % % % % % % % % % % % % % % % % % % 

In this section, we extend the algebraic theory of spinors and exterior forms of Section \ref{sec:algebraic_spinors} to bundles of real irreducible Clifford modules equipped with an arbitrary connection and an admissible bilinear pairing. This allows us to define and study parallel spinors on pseudo-Riemannian manifolds of signature $(4,3)$. 

% % % % % % % % % % % % % % % % % % % % % % % % % % % % % % % % % % % % % % 

\subsection{Preliminaries}

% % % % % % % % % % % % % % % % % % % % % % % % % % % % % % % % % % % % % % 

Let $(M,g)$ be a connected and oriented pseudo-Riemannian manifold of signature $(4,3)$. We denote by $\mathrm{Cl}(M,g)$ the bundle of real Clifford algebras of the cotangent bundle $(T^*M,g^*)$, which is modeled on the real Clifford algebra $\mathrm{Cl}(4,3)$.

\begin{definition}
A \emph{bundle of real Clifford modules} on $(M,g)$ is a pair $(S,\Gamma)$, where $S$ is a real vector bundle and $\Gamma:\mathrm{Cl}(M,g)\to\End(S)$ is a unital morphism of bundles of unital and associative real algebras.
\end{definition}

\noindent
For each point $m\in M$ the unital morphism of associative algebras $\Gamma_m:\mathrm{Cl}(T^*_mM,g^*_m)\to\End(S_m)$ defines a Clifford module denoted by $(S_m,\Gamma_m)$. There exists a Clifford module $(\Sigma,\gamma)$, unique up to isomorphism in the category of \emph{unbased} Clifford algebras \cite{LS19,LS18}, such that $(S_m,\Gamma_m)\cong(\Sigma,\gamma)$ for every $m\in M$. We say then that $(S,\Gamma)$ is a bundle of real Clifford modules of \emph{type} $(\Sigma,\gamma)$.

\begin{definition}\textcolor{white}{}
\begin{itemize}
\itemsep 0em
\item A \emph{paired spinor bundle of type $(\Sigma,\gamma)$} is a triple $(S,\Gamma,\Scr B)$ consisting of a bundle of real Clifford modules $(S,\Gamma)$ of type $(\Sigma,\gamma)$ equipped with an admissible bilinear pairing $\Scr B$ \cite{AC97,ACDVP05}. An \emph{irreducible paired spinor bundle} is a paired spinor bundle $(S,\Gamma,\Scr B)$ of irreducible type $(\Sigma,\gamma)$.
\item Let $(S,\Gamma,\Scr B)$ be an irreducible paired spinor bundle. A global section $\varepsilon\in\Gamma(S)$ is an \emph{irreducible spinor} on $(M,g)$.
\item Let $(S,\Gamma,\Scr B)$ be a paired spinor bundle of $(M,g)$ and let $\cD$ be a connection on $S$. A section $\varepsilon\in\Gamma(S)$ is a \emph{parallel spinor} on $(M,g)$ relative to $\cD$ if $$\cD\varepsilon=0.$$
\end{itemize}
\end{definition}

\begin{remark}
As stated in the previous definition, we have chosen the term \emph{parallel spinor} to refer to a spinor parallel under a general connection on the spinor bundle. Given the plethora of names, such as parallel spinors, Killing spinors, Codazzi spinors, Cauchy spinors, skew-torsion parallel spinors, skew-Killing spinors, or generalized Killing spinors, used in the literature to refer to spinors parallel under various connections, we concluded it was necessary to coin a simple and intuitive term to refer to a spinor parallel under any given connection on the spinor. Although the term \emph{parallel spinor} with no further decorations usually refers to a spinor parallel under the Levi-Civita connection, we felt it was unnecessary to introduce more complicated terminology.  
\end{remark}

\begin{remark}
It is guaranteed that an admissible bilinear pairing $\Scr B$ exists on an irreducible spinor bundle $(S,\Gamma)$ if $(M,g)$ is \emph{strongly spin}, that is if every spin structure on $(M,g)$ admits a reduction to the identity component of the spin group.
\end{remark}

\noindent
Let $(S,\Gamma,\Scr B)$ be an irreducible paired spinor bundle in signature $(4,3)$, which thus satisfies $p-q\equiv_81$. Then, by \cite{LS18,LS19} the triple $(S,\Gamma,\Scr B)$ is associated to a spin structure and we can write $\cD=\nabla-\cA$ for a unique element $\cA\in\Omega^1(\End(S))$, where $\nabla$ is the lift of the Levi-Civita connection to $S$. In this case, the equation satisfied by the parallel spinor can be equivalently written as: 
\begin{equation}
\label{eq:diff_spinor}
\nabla\varepsilon=\cA(\varepsilon).
\end{equation}

\noindent
Solutions to \eqref{eq:diff_spinor} are called parallel spinors \emph{relative} to $\cA$ of type $(\Sigma,\gamma_l)$. Recall that associated to $(S,\Gamma)$ we obtain an isomorphism of bundles of unital and associative algebras:
\begin{equation*}
\Psi_\Gamma^<:(\Lambda^<T^*M,\vee)\to\End(S)
\end{equation*}

\noindent
as well as a \emph{quadratic map}:
\begin{equation*}
  \cE^\mu_\Gamma:S\to\Lambda^<T^*M  
\end{equation*}

\noindent
which we extend pointwise to sections $\Gamma(S)$ of $S$ and denote by the same symbol. Evaluating $\cA$ on a vector field $v\in\fX(M)$ we obtain a field of endomorphisms $\cA_v\in\Gamma(\End(S))$ to which we can apply the inverse of $\Psi^<_\Gamma$ in order to obtain a section of $\Lambda^<T^*M$. We then define:
\begin{equation*}
\fra_v:=(\Psi^<_\Gamma)^{-1}(\cA_v),\quad v\in\fX(M)
\end{equation*}

\noindent
and we denote by $\fra\in\Omega^1(\Lambda^<T^*M)$ the associated one-form taking values in $\Lambda^<T^*M$. We refer to $\fra$ as the \emph{dequantization} of $\cA$, and for simplicity in the exposition, we will sometimes refer to solutions of \eqref{eq:diff_spinor} as parallel spinors relative to $\fra$, instead of $\cA$. These are not equivalent, since defining $\cA$ requires specifying an underlying spinor bundle, while referring to $\fra$ does not, as the latter is simply a one-form valued exterior form. Indeed, there could be one-form valued endomorphisms of different spinor bundles whose \emph{dequantization} gives the same one-form valued exterior form $\fra$. As we will see below in Proposition \ref{prop:generaldifferentialspinor}, this potential ambiguity in $\cA$ is irrelevant for studying parallel spinors through their square since only the dequantization $\fra$ appears in Equation \eqref{eq:master_formula}. The following proposition, which we will apply extensively in the following, is a direct application of \cite[Theorem 3.6]{Sha24} to signature $(4,3)$.

\begin{prop} 
\label{prop:generaldifferentialspinor}
    A strongly spin pseudo-Riemannian manifold $(M,g)$ of signature $(4,3)$ admits a parallel spinor of type $(\Sigma,\gamma_l)$ relative to $\cA$ with given dequantization $\fra\in \Omega^1(\Lambda^<T^*M)$ if and only if there exists a nowhere vanishing exterior form $\alpha\in\Gamma(\Lambda^<T^*M)$ such that the following differential system is satisfied: \begin{equation}\label{eq:master_formula}
        \nabla\alpha=\fra\vee\alpha+\alpha\vee(\pi\circ\tau)(\fra)
    \end{equation} and either the following algebraic equations are satisfied for every exterior form $\beta\in\Gamma(\Lambda^<T^*M)$: $$(\pi\circ\tau)(\alpha)= \alpha,\quad\alpha\vee\beta\vee\alpha=8(\alpha\vee\beta)^{(0)}\alpha$$ or, equivalently, the following equations: $$\alpha\vee\alpha=8\alpha^{(0)}\alpha,\quad(\pi \circ\tau)(\alpha)= \alpha,\quad\alpha\vee\beta\vee\alpha=8(\alpha\vee\beta)^{(0)}\alpha$$ are satisfied for some fixed exterior form $\beta\in\Gamma(\Lambda^<T^*M)$ such that $(\alpha\vee\beta)^{(0)}\neq0$.
\end{prop}

\noindent
Now we consider the most general parallel spinor on a strongly spin pseudo-Riemannian manifold $(M,g)$ of signature $(4,3)$, obtaining a general differential system for its associated spinorial exterior form. We know by Proposition \ref{prop:algebraic_characterization} that an exterior form $\alpha$ is the square of a spinor only if $\alpha=f+\phi$ for $f\in\cC^\infty(M)$ and $\phi\in\Omega^3(M)$. We use this result together with Proposition \ref{prop:generaldifferentialspinor} to characterize when $\alpha$ corresponds to a parallel spinor relative to $\cA\in\Omega^1(\End(S))$ with dequantization $\fra\in\Omega^1(\Lambda^<T^*M)$.

\begin{prop}
\label{prop:intermediate_formula}
A strongly spin pseudo-Riemannian manifold $(M,g)$ of signature $(4,3)$ admits a parallel spinor of type $(\Sigma,\gamma_l)$ relative to $\fra\in \Omega^1(\Lambda^<T^*M)$ if and only if the following system is satisfied for every $v\in \mathfrak{X}(M)$:
\begin{align}
\nabla_vf&=2f\fra_v^{(0)}-2\escal{\fra_v^{(3)},\phi}, \label{eq:generalsystem1}\\
\nabla_v\phi&=2f\fra_v^{(3)}+2\fra_v^{(0)}\phi+2l*(\fra_v^{(1)}\wedge\phi)-2\fra_v^{(2)}\triangle_1\phi+2l*(\fra_v^{(3)}\triangle_1\phi),
\label{eq:generalsystem2}
\end{align} 
where $f+\phi\in\Im(\cE^{+}_{\Gamma})\subset\Gamma(\Lambda^<T^*M)$ for some irreducible paired spinor bundle $(S,\Gamma,\cB)$ on $(M,g)$.
\end{prop}

\begin{proof}
    For every vector field $v\in\fX(M)$ we have $\fra_v=\fra_v^{(0)}+\fra_v^{(1)}+\fra_v^{(2)}+\fra_v^{(3)}$, where $\fra_v^{(j)}\in\Omega^j(M)$. The right-hand side of Equation \eqref{eq:master_formula} becomes \begin{align*}
        \fra_v\vee\alpha+\alpha\vee(\pi\circ\tau)(\fra_v)&=(\fra_v^{(0)}+\fra_v^{(1)}+\fra_v^{(2)}+\fra_v^{(3)})\vee(f+\phi)+(f+\phi)\vee(\fra_v^{(0)}-\fra_v^{(1)}-\fra_v^{(2)}+\fra_v^{(3)})\\
        &=2f\fra_v^{(0)}+2f\fra_v^{(3)}+2\fra_v^{(0)}\phi\\
        &\quad+(\fra_v^{(1)}\vee\phi-\phi\vee\fra_v^{(1)})+(\fra_v^{(2)}\vee\phi-\phi\vee\fra_v^{(2)})+(\fra_v^{(3)}\vee\phi+\phi\vee\fra_v^{(3)})\\
        &=2f\fra_v^{(0)}+2f\fra_v^{(3)}+2\fra_v^{(0)}\phi\\
        &\quad+2l*(\fra_v^{(1)}\wedge\phi)-2\fra_v^{(2)}\triangle_1\phi+2l*(\fra_v^{(3)}\triangle_1\phi)-2\escal{\fra_v^{(3)},\phi}.
    \end{align*} Comparing degrees we find that Equation \eqref{eq:master_formula} becomes the system of the statement.
\end{proof}

% % % % % % % % % % % % % % % % % % % % % % % % % % % % % % % % % % % % % % 

\subsection{Parallel spinors and \texorpdfstring{$\mathrm{G}^{\ast}_2$}{G2*}-structures}
\label{subsec:diff_spinor_G2*}

% % % % % % % % % % % % % % % % % % % % % % % % % % % % % % % % % % % % % % 

Given a spinor $\varepsilon\in\Gamma(S)$ of non-zero norm we can decompose $S=\escal{\varepsilon}\oplus\escal{\varepsilon}^\perp$ where $\escal{\varepsilon}^\perp\subset S$ denotes the orthogonal complement of $\varepsilon$ with respect to $\cB$. We have the isomorphism:
\begin{equation*}
\escal{\varepsilon}^\perp\cong\{v\cdot\varepsilon\mid v\in\fX(M)\},
\end{equation*}

\noindent
see Remark \ref{remark:decomposition_Sigma}. This implies in particular that for every $v\in\fX(M)$, the spinor $\nabla_v\varepsilon\in\Gamma(S)$ can be written as:
\begin{equation*}
\nabla_v\, \varepsilon=\theta(v)\varepsilon+A(v)\cdot\varepsilon
\end{equation*}

\noindent
for some $\theta\in\Omega^1(M)$ and $A\in\End(TM)$. Hence, for a parallel spinor $\varepsilon$ relative to $\fra\in\Omega^1(\Lambda^<T^*M)$, it must be possible to find an element $\bar\fra\in\Omega^1(\Lambda^<T^*M)$ with $\bar\fra^{(2)}=\bar\fra^{(3)}=0$ and relative to which $\varepsilon$ is again a parallel spinor. This is verified in the following result, which characterizes parallel spinors of nowhere vanishing and possibly non-constant norm.

\begin{thm}
\label{thm:master_formula}
A strongly spin pseudo-Riemannian manifold $(M,g)$ of signature $(4,3)$ admits a nowhere isotropic parallel spinor $\varepsilon$ of type $(\Sigma,\gamma_l)$ relative to $\fra\in \Omega^1(\Lambda^<T^*M)$ if and only if there exists a nowhere vanishing three-form $\phi\in\Omega^3(M)$ satisfying the algebraic equation:
\begin{equation}
\label{eq:algebraicphinonisotropic}
  6\,\kappa \sqrt{-\frac{\escal{\phi,\phi}}{7}}\,\phi=l*(\phi\triangle_1\phi)\, ,  \qquad P_{(27)} (\fra_v) =  0 \, , \qquad \forall\,\, v\in\mathfrak{X}(M),
\end{equation}

\noindent
together with the differential system:
\begin{align*}
\nabla_v\phi&=2\bar\fra_v^{(0)}\phi+2l*(\bar\fra_v^{(1)}\wedge\phi)\, , \qquad \forall \,\, v\in\fX(M),
\end{align*}

\noindent
where:
\begin{equation*}
\bar\fra_v^{(0)}=\fra_v^{(0)}+7f\kappa_v^{(0)}\quad\text{and}\quad\bar\fra_v^{(1)}=\fra_v^{(1)}+3f\sigma_v^{(1)}+4lf\kappa_v^{(1)}
\end{equation*}

\noindent
and where $P_{(27)}\colon \Omega^3(M) \to \Omega^3_{27}(M)$ is the canonical projection with respect to the $\G_2^{\ast}$-structure defined by $\varepsilon$.
\end{thm}

\begin{remark}
This result is closely related to the spinorial description of Riemannian $\G_2$-structures given in \cite{ACFH15}. Aside from the difference in signature, Theorem \ref{thm:master_formula} diverges from the description given in Op.\ Cit.\ in that we do not consider the spinor to be of constant norm, we consider a general parallel spinor, which leads to the obstruction appearing in the second equation of \eqref{eq:algebraicphinonisotropic}, and our result includes the necessary and sufficient algebraic condition contained in the first equation of \eqref{eq:algebraicphinonisotropic}.
\end{remark}

\begin{proof}
By Theorem \ref{thm:stabilizerG2*} and Proposition \ref{prop:intermediate_formula}, $(M,g)$ admits a nowhere isotropic parallel spinor relative to $\fra\in \Omega^1(\Lambda^<T^*M)$ if and only if it admits a three-form $\phi$ satisfying the algebraic relation \eqref{eq:algebraicphinonisotropic} as well as the differential system given by \eqref{eq:generalsystem1} and \eqref{eq:generalsystem2} with:
\begin{equation*}
f = \sqrt{-\frac{\escal{\phi,\phi}}{7}}.
\end{equation*}

\noindent
Since $\varepsilon$ is nowhere isotropic, it defines a $\G^{\ast}_2$-structure respect to which we consider the decomposition of $\fra_v^{(2)}$ and $\fra_v^{(3)}$ into irreducible components, that is: 
\begin{align*}
\fra_v^{(2)}&=\iota_{(\sigma_v^{(1)})^\sharp}\phi+\sigma_v^{14}\in\Omega^2_7\oplus\Omega^2_{14},\\
\fra_v^{(3)}&=\kappa_v^{(0)}\phi+*(\kappa_v^{(1)}\wedge\phi)+\kappa_v^{27}\in\Omega^3_1\oplus\Omega^3_7\oplus\Omega^3_{27}.
\end{align*}

\noindent
Using Lemma \ref{lemma:aux-triangle1} and isolating the terms that belong to the same irreducible $\G^{\ast}_2$ component, we obtain:
\begin{equation*}
\nabla_v\phi=-4f\kappa_v^{27}+2\fra_v^{(0)}\phi+14f\kappa_v^{(0)}\phi+2l*(\fra_v^{(1)}\wedge\phi)+6lf*(\sigma_v^{(1)}\wedge\phi)+8f*(\kappa_v^{(1)}\wedge\phi).
\end{equation*}

\noindent
Hence $\kappa_v^{(27)}=0$, $\bar\fra_v^{(0)}=\fra_v^{(0)}+7f\kappa_v^{(0)}$ and $\bar\fra_v^{(1)}=\fra_v^{(1)}+3f\sigma_v^{(1)}+4lf\kappa_v^{(1)}$. This shows that, given $\kappa_v^{(27)}=0$, the differential system given in \eqref{eq:generalsystem2} is equivalent to:
\begin{align*}
\nabla_vf&=2f\bar\fra_v^{(0)},\\
\nabla_v\phi&=2\bar\fra_v^{(0)}\phi+2l*(\bar\fra_v^{(1)}\wedge\phi).
\end{align*}

\noindent
Setting $f = \sqrt{-\frac{\escal{\phi,\phi}}{7}}$ as required by Theorem \ref{thm:stabilizerG2*}, it follows that the first equation is a consequence of the second and therefore we conclude.
\end{proof}

\noindent
By the previous theorem, we immediately obtain the following \emph{obstruction} result.
\begin{cor}
Let $\fra\in \Omega^1(\Lambda^<T^*M)$ and let $\varepsilon$ be a nowhere isotropic and nowhere vanishing spinor such that $P_{(27)} (\fra_v) \neq 0$ for every $v\in\mathfrak{X}(M)$. Then, $\varepsilon$ is not a parallel spinor relative to $\fra$.
\end{cor}

\noindent
It is well-known that the metric induced by the three-form $\phi$ has holonomy contained in $\mathrm{G}_2^*$ if $\nabla\phi=0$ (see e.g.\ \cite[Proposition 2.11]{Fre13}). Theorem \ref{thm:master_formula} shows that this can happen for parallel spinors relative to non-trivial $\cA$ and gives an explicit formula in terms of the dequantization of the latter.  

\begin{cor}
Let $\varepsilon$ be a pseudo-unit parallel spinor on $(M,g)$ relative to $\fra\in \Omega^1(\Lambda^{<} T^{\ast} M)$. Then, the $\mathrm{G}_2^*$-structure $\phi$ associated to $\varepsilon$ has holonomy contained in $\mathrm{G}_2^*$ if $\fra_v^{(0)}+7\kappa_v^{(0)} = 0$ and $\fra_v^{(1)}+3\sigma_v^{(1)}+4l\kappa_v^{(1)} = 0$ for every $v\in\mathfrak{X}(M)$.
\end{cor}

% % % % % % % % % % % % % % % % % % % % % % % % % % % % % % % % % % % % % % 

\subsection{Isotropic parallel spinors}
\label{subsec:null_diff_spinors}

% % % % % % % % % % % % % % % % % % % % % % % % % % % % % % % % % % % % % % 

It follows from Theorem \ref{thm:null_spinor} that a nowhere vanishing exterior form $\alpha\in \Gamma(\Lambda^{<} T^{\ast}M)$ is the square of an isotropic spinor if and only if there exists a good open cover $\left\{ U_a \right\}_{a\in I}$ and an associated family of orthogonal and isotropic triples of one-forms $\{\vartheta_a,\theta_a,\eta_a\}_{a\in I}$ such that:
\begin{equation*}
\alpha|_{U_a}=\vartheta_a\wedge\theta_a\wedge\eta_a.
\end{equation*}

\noindent
This motivates the following definition.

\begin{definition}
\label{def:isotropicfamily}
An \emph{isotropic family of triples of one-forms} on a good open cover $\left\{ U_a \right\}_{a\in I}$ is a family of nowhere vanishing orthogonal and isotropic triples of one-forms $\{\vartheta_a,\theta_a,\eta_a\}_{a\in I}$ with $\vartheta_a,\theta_a,\eta_a\in\Omega^1(U_a)$. Such family is \emph{coherent} if:
\begin{equation*}
    (\vartheta_a\wedge\theta_a\wedge\eta_a)\vert_{U_{ab}} = (\vartheta_b\wedge\theta_b\wedge\eta_b)\vert_{U_{ab}}
\end{equation*}

\noindent
on every non-trivial overlap $U_{ab} := U_a\cap U_b$.
\end{definition}

\noindent
Clearly, a coherent family of triples of one-forms always gives rise to a well-defined global three-form on $M$, which we will denote generically by $\phi\in\Omega^3(M)$. Hence, as a consequence of Theorem \ref{thm:null_spinor}, the existence of isotropic spinors can be reformulated in the following terms.

\begin{cor}
A strongly spin pseudo-Riemannian manifold $(M,g)$ of signature $(4,3)$ admits a nowhere vanishing isotropic spinor if and only if it admits an isotropic and coherent family of triples of one-forms.
\end{cor}

\begin{proof}
The result follows from the standard fact that a smooth form is pointwise decomposable, as we know by Theorem \ref{thm:null_spinor} the square of an isotropic spinor is, if and only if is locally decomposable in terms of local smooth one-forms.
\end{proof}

\noindent
Every coherent family defines an orientable vector subbundle $W\subset T^{\ast}M$ with volume form induced by $\phi$. This rank-three vector bundle $W$ is topologically trivial if and only if $\phi$ can be globally written as $\phi = \vartheta\wedge\theta\wedge\eta$ in terms of three isotropic and mutually orthogonal one-forms $\vartheta , \theta , \eta \in \Omega^1(M)$.  

\begin{thm}\label{thm:null_diff_spinor}
Let $(M,g)$ be a strongly spin pseudo-Riemannian manifold of signature $(4,3)$. Then it admits an isotropic parallel spinor of type $(\Sigma,\gamma_l)$ relative to $\fra\in \Gamma(\Lambda^{<} T^{\ast}M)$ if and only if it admits an isotropic and coherent family $\{\vartheta_a,\theta_a,\eta_a\}$ of triples of one-forms relative to some good open cover $\left\{ U_a \right\}_{a\in I}$ and such that 
\begin{align*}
        \escal{\fra_v^{(3)},\vartheta_a\wedge\theta_a\wedge\eta_a}&=0,\\
        \nabla_v(\vartheta_a\wedge\theta_a\wedge\eta_a)&=2\fra_v^{(0)}\vartheta_a\wedge\theta_a\wedge\eta_a+2l*(\fra_v^{(1)}\wedge\vartheta_a\wedge\theta_a\wedge\eta_a)\\
        &\quad-2(\fra_v^{(2)}\triangle_1\vartheta_a)\wedge\theta_a\wedge\eta_a+2(\fra_v^{(2)}\triangle_1\theta_a)\wedge\vartheta_a\wedge\eta_a-2(\fra_v^{(2)}\triangle_1\eta_a)\wedge\vartheta_a\wedge\theta_a\\
        &\quad+2l*((\fra_v^{(3)}\triangle_1\vartheta_a)\wedge\theta_a\wedge\eta_a-(\fra_v^{(3)}\triangle_1\theta_a)\wedge\vartheta_a\wedge\eta_a+(\fra_v^{(3)}\triangle_1\eta_a)\wedge\vartheta_a\wedge\theta_a)
    \end{align*} 
    for all $v\in\fX(M)$.
\end{thm}

\begin{proof}
    This is an application of Proposition \ref{prop:intermediate_formula}. It only remains to prove the new expressions for $\fra_v^{(2)}\triangle_1(\vartheta_a\wedge\theta_a\wedge\eta_a)$ and $\fra_v^{(3)}\triangle_1(\vartheta_a\wedge\theta_a\wedge\eta_a)$. First we compute $$\iota_{e_i}(\vartheta_a\wedge\theta_a\wedge\eta_a)=\vartheta_a(e_i)\theta_a\wedge\eta_a-\theta_a(e_i)\vartheta_a\wedge\eta_a+\eta_a(e_i)\vartheta_a\wedge\theta_a.$$ Then we have \begin{align*}
        \fra_v^{(2)}\triangle_1(\vartheta_a\wedge\theta_a\wedge\eta_a)&=\iota_{e_i}\fra_v^{(2)}\wedge\iota_{e_i}(\vartheta_a\wedge\theta_a\wedge\eta_a)\\
        &=(\fra_v^{(2)}(e_i)\vartheta_a(e_i))\wedge\theta_a\wedge\eta_a-(\fra_v^{(2)}(e_i)\theta_a(e_i))\wedge\vartheta_a\wedge\eta_a+(\fra_v^{(2)}(e_i)\eta_a(e_i))\wedge\vartheta_a\wedge\theta_a\\
        &=(\fra_v^{(2)}\triangle_1\vartheta_a)\wedge\theta_a\wedge\eta_a-(\fra_v^{(2)}\triangle_1\theta_a)\wedge\vartheta_a\wedge\eta_a+(\fra_v^{(2)}\triangle_1\eta_a)\wedge\vartheta_a\wedge\theta_a.
    \end{align*} To obtain $\fra_v^{(3)}\triangle_1(\vartheta_a\wedge\theta_a\wedge\eta_a)$ we proceed in an analogous way and hence we conclude.
\end{proof}

\noindent
As an application, we apply the previous general result to spinors parallel under a metric connection $\nabla^{g,A}$ with \emph{contorsion} $A\in\Gamma(T^{\ast}M\otimes \Lambda^2 T^{\ast}M)$. Such a connection can always be written as follows by definition:
\begin{equation*}
\nabla^{g,A} = \nabla^g +  A
\end{equation*}

\noindent
in terms of the Levi-Civita connection $\nabla^g$ of the underlying metric $g$ and the \emph{contorsion tensor} $A\in\Gamma(T^{\ast}M\otimes \Lambda^2 T^{\ast}M)$ of the connection. Since $(M,g)$ is assumed to be strongly spin and $\nabla^{g,A}$ is metric, it lifts to any irreducible paired spinor bundle $(S,\Gamma,\cB)$ on $(M,g)$, which we denote by the same symbol for ease of notation. Spinors $\varepsilon\in \Gamma(S)$ satisfying:
\begin{equation*}
\nabla^{g,A}\varepsilon = 0 \quad \Leftrightarrow \quad \nabla^g_v \varepsilon  = - \Psi^{<}_{\Gamma}(A_v) (\varepsilon)\, , \quad \forall \,\, v\in \mathfrak{X}(M)
\end{equation*}

\noindent
are called \emph{torsion parallel spinors} and define a particular class of parallel spinors on $(M,g)$ relative to data $\fra\in \Omega^1(\Lambda^<T^*M)$ satisfying $\fra^{(0)}=\fra^{(1)}=\fra^{(3)}=0$.

\begin{cor}
\label{cor:spinor_with_torsion}
A strongly spin pseudo-Riemannian manifold $(M,g)$ of signature $(4,3)$ admits a torsion parallel spinor with contorsion $A\in\Gamma(T^{\ast}M\otimes \Lambda^2 T^{\ast}M)$ if and only if it admits an isotropic and coherent family of triples of one-forms $\{\vartheta_a,\theta_a,\eta_a\}$ preserved by $\nabla^{g,A}$.
\end{cor}

\begin{remark}
We can think of this corollary as an extension of \cite[Proposition 8.7]{KathHabilitation} in signature $(4,3)$ to the case of a metric connection with torsion, without the need to impose the underlying manifold to be simply connected.
\end{remark}

\begin{proof}
Applying Theorem \ref{thm:null_diff_spinor} to the case where $\fra^{(0)}=\fra^{(1)}=\fra^{(3)}=0$ we obtain: 
\begin{equation}
\label{eq:nabla_null_form}
\nabla^g_v(\vartheta_a\wedge\theta_a\wedge\eta_a)=-2(\fra_v^{(2)}\triangle_1\vartheta_a)\wedge\theta_a\wedge\eta_a+2(\fra_v^{(2)}\triangle_1\theta_a)\wedge\vartheta_a\wedge\eta_a-2(\fra_v^{(2)}\triangle_1\eta_a)\wedge\vartheta_a\wedge\theta_a.
\end{equation}

\noindent
Let $\alpha_a:=\vartheta_a\wedge\theta_a\wedge\eta_a$. Note that $0=\nabla^g_v(\alpha_a\wedge\vartheta_a)=(\nabla^g_v\alpha_a)\wedge\vartheta_a+\alpha_a\wedge(\nabla^g_v\vartheta_a)$. Multiplying both sides of \eqref{eq:nabla_null_form} by $\vartheta_a$ we obtain $(\nabla^g_v\alpha_a)\wedge\vartheta_a=-2(\fra_v^{(2)}\triangle_1\vartheta_a)\wedge\alpha_a=-2\fra_v^{(2)}(\vartheta_a^\sharp)\wedge\alpha_a$. Hence:
\begin{equation*}
    \nabla^g_v\vartheta_a=-2\fra_v^{(2)}(\vartheta_a^\sharp)+\zeta_v^{\vartheta_a},
\end{equation*}

\noindent
where $\zeta_v^{\vartheta_a}$ is a one-form such that $\zeta_v^{\vartheta_a}\wedge\alpha_a=0$ for all $v\in\mathfrak{X}(M)$, thus $\zeta_v^{\vartheta_a}$ is in the span of $\{\vartheta_a,\theta_a,\eta_a\}$. Set $A:=2\fra^{(2)}$. Then:
\begin{equation*}
\nabla^{g,A}_v\vartheta_a=\nabla^g_v\vartheta_a+A_v(\vartheta_a^\sharp)=\zeta_v^{\vartheta_a}\in\Span\{\vartheta_a,\theta_a,\eta_a\}.
\end{equation*}

\noindent
Similarly, we get: 
\begin{align*}
\nabla^{g,A}_v\theta_a &=\nabla^g_v\theta_a+A_v(\theta_a^\sharp)=\zeta_v^{\theta_a}\in\Span\{\vartheta_a,\theta_a,\eta_a\},\\
\nabla^{g,A}_v\eta_a &=\nabla^g_v\eta_a+A_v(\eta_a^\sharp)=\zeta_v^{\eta_a}\in\Span\{\vartheta_a,\theta_a,\eta_a\}
\end{align*}

\noindent
and thus we conclude.
\end{proof}

\noindent
Note that the contorsion tensor $A\in\Gamma(T^{\ast}M\otimes \Lambda^2 T^{\ast}M)$ of a metric connection $\nabla^{g,A}$ is related to its torsion $T\in \Gamma(\Lambda^2 T^{\ast}M \otimes T^*M)$ by the formula $T(v_1,v_2)=A(v_1,v_2)-A(v_2,v_1)$, where $v_1,v_2 \in \mathfrak{X}(M)$. If $\nabla^{g,A}$ has totally skew-symmetric torsion, i.e.\ $T\in\Omega^3(M)$, then $2A=  T$. As a particular case of the previous corollary, we can take $\nabla^{g,A}$ to be the Levi-Civita connection by setting $A=0$. In this case, we have a \emph{degenerate} analog of quaternionic Kähler geometry, but instead of having a bundle of non-degenerate local two-forms, we have a bundle of isotropic and orthogonal local one-forms.

% % % % % % % % % % % % % % % % % % % % % % % % % % % % % % % % % % % % % % 

\subsection{Metrics admitting isotropic parallel spinors}
\label{subsec:examples}

% % % % % % % % % % % % % % % % % % % % % % % % % % % % % % % % % % % % % % 

Corollary \ref{cor:spinor_with_torsion} can be immediately applied to construct explicit examples of metrics $g$ of signature $(4,3)$ on $\mathbb{R}^7$ admitting irreducible isotropic real spinors parallel with respect to a metric connection with torsion. Consider $\mathbb{R}^7$ with standard coordinates $(x_1,y_1 , x_2, y_2 , x_3, y_3 ,z)$. Inspired by the structure of three-dimensional Lorentzian manifolds admitting a special isotropic vector field, more specifically the so-called Kundt three-manifolds \cite{BMZ22}, we propose the following ansatz for a metric of signature $(4,3)$ on $\mathbb{R}^7$:
\begin{equation}
\label{eq:metricexample}
g = \sum_{i=1}^3 ( H_i \dd x_i \otimes \dd x_i + e^{K_i}\dd x_i \odot \dd y_i) + e^{F}\dd z\otimes \dd z
\end{equation}

\noindent
for arbitrary functions $H_1 , H_2 , H_3, K_1 , K_2 , K_3, F \in\mathcal{C}^{\infty}(\mathbb{R}^7)$. Roughly speaking, each factor of the form $H_i \dd x_i \otimes \dd x_i + e^{K_i}\dd x_i \odot \dd y_i$ occurs naturally in the local form of a Lorentzian metric admitting a parallel vector field of the form $\partial_{y_i}$ and gives a contribution of signature $(1,1)$ to the Lorentzian metric, which thus has a positive definite \emph{transverse} part. Hence, summing three of these factors plus a \emph{transverse} element of the form $e^{F}\dd z\otimes \dd z$ we obtain a metric of signature $(4,3)$ which is naturally equipped with three coordinate vector fields which are all isotropic and mutually orthogonal. Metric \eqref{eq:metricexample} also provides a natural generalization of the most general local metric admitting a parallel (under Levi-Civita) isotropic real spinor, which was obtained in \cite[Theorem 8.2]{KathHabilitation}, see also \cite{Bryant00}. As candidates for the metric dual of the triple of one-forms as described in Corollary \ref{cor:spinor_with_torsion}, we set:
\begin{equation*}
\vartheta^{\sharp_g} := \partial_{y_1} \, , \qquad \theta^{\sharp_g} :=  \partial_{y_2}\, , \qquad \eta^{\sharp_g} :=\partial_{y_3}
\end{equation*}

\noindent
and thus:
\begin{equation}
\label{eq:tripleexamples}
\vartheta = e^{K_1} \dd x_1 \, , \qquad \theta = e^{K_2} \dd x_2 \, , \qquad \eta = e^{K_3} \dd x_3.
\end{equation}

\noindent
Let $\nabla^{g,A}$ be a metric connection with contorsion $A$, and let $\varepsilon\colon \mathbb{R}^7 \to \mathbb{R}^8$ be a spinor whose associated exterior form is $\phi = \vartheta\wedge\theta\wedge \eta$ as described above. Then, by Corollary \ref{cor:spinor_with_torsion} it follows that $\nabla^{g,A}\varepsilon = 0$ if and only if:
\begin{eqnarray}
\label{eq:involutivecondition}
\nabla^{g,A}_w \dd x_i \in \Span_{\mathcal{C}^{\infty}(\mathbb{R}^7)}\{\dd x_1, \dd x_2, \dd x_3\}\, , \qquad i=1,2,3
\end{eqnarray}

\noindent
for every $w\in \mathfrak{X}(\mathbb{R}^7)$. We have:
\begin{equation*}
(\nabla^{g,A}_{w_1} \dd x_i)(w_2) = (\nabla^{g}_{w_1} \dd x_i)(w_2) + e^{-K_i} A(w_1 , \partial_{y_i}, w_2)
\end{equation*}

\noindent
and therefore this equation is equivalent to:
\begin{equation}
\label{eq:existencespinorexample}
(\nabla^{g}_{w_1} \dd x_i)(\partial_{z}) + e^{-K_i} A(w_1 , \partial_{y_i}, \partial_{z}) = 0\, , \quad (\nabla^{g}_{w_1} \dd x_i)(\partial_{y_j}) + e^{-K_i} A(w_1 , \partial_{y_i}, \partial_{y_j}) = 0\, ,  \quad i,j = 1,2,3.
\end{equation}

\noindent
A computation gives:
\begin{align*}
 2 \nabla^g \dd x_i &= e^{K_i} \frac{\partial H_i}{\partial y_i} \dd x_i \otimes \dd x_i -   \sum_j \frac{\partial K_i}{\partial x_j} \dd x_i \odot \dd x_j  -   \sum_{j\neq i} \frac{\partial K_i}{\partial y_j} \dd x_i \odot \dd y_j -  \frac{\partial K_i}{\partial z} \dd x_i \odot \dd z\\
&\quad + e^{-K_i} \sum_{j\neq i} \frac{\partial H_j}{\partial y_i} \dd x_j \otimes \dd x_j  + e^{-K_i} \sum_{j\neq i} \frac{\partial e^{K_j}}{\partial y_i} \dd x_j \odot \dd y_j  + e^{-K_i} \frac{\partial e^{F}}{\partial y_i} \dd z \otimes \dd z
\end{align*}

\noindent
for every fixed $i = 1,2,3$ (no summation is understood for the repeated index $i$). Hence:
\begin{equation*}
2 (\nabla^g \dd x_i)(\partial_{y_j}) = (1-\delta_{ij}) ( e^{-K_i} \frac{\partial e^{K_j}}{\partial y_i} \dd x_j - \frac{\partial K_i}{\partial y_j} \dd x_i)\, , \quad   2 (\nabla^g \dd x_i)(\partial_z) =   -  \frac{\partial K_i}{\partial z} \dd x_i + e^{-K_i} \frac{\partial e^{F}}{\partial y_i} \dd z.
\end{equation*}

\noindent
Comparing with Equation \eqref{eq:existencespinorexample}, we conclude that $\mathbb{R}^7$ equipped with the metric \eqref{eq:metricexample} admits a parallel irreducible isotropic spinor with associated triple of one-forms  \eqref{eq:tripleexamples} if and only if:
\begin{eqnarray*}
& -  \frac{\partial K_i}{\partial z} \dd x_i + e^{-K_i} \frac{\partial e^{F}}{\partial y_i} \dd z + \frac{1}{2} e^{-K_i} A(- , \partial_{y_i}, \partial_{z}) = 0,\\
& (1-\delta_{ij}) ( e^{-K_i} \frac{\partial e^{K_j}}{\partial y_i} \dd x_j - \frac{\partial K_i}{\partial y_j} \dd x_i) + \frac{1}{2} e^{-K_i} A(- , \partial_{y_i}, \partial_{y_j}) = 0 
\end{eqnarray*}

\noindent
for every $i,j = 1,2,3$. Equivalently:
\begin{eqnarray}
\label{eq:torsionexample}
\frac{1}{2} A(- , \partial_{y_i}, \partial_{z}) = \frac{\partial e^{K_i}}{\partial z} \dd x_i -  \frac{\partial e^{F}}{\partial y_i} \dd z \, , \qquad \frac{1}{2}  A( - , \partial_{y_i}, \partial_{y_j}) = \frac{\partial e^{K_i} }{\partial y_j} \dd x_i - \frac{\partial e^{K_j}}{\partial y_i} \dd x_j\, , \quad i\neq j.
\end{eqnarray}

\noindent
Rather than as differential equations for the coefficients of $g$, these conditions can be understood as determining certain components of $A$ to guarantee that there exists an irreducible isotropic spinor $\varepsilon \colon \mathbb{R}^7 \to \mathbb{R}^8$ such that $\nabla^{g,A}\varepsilon = 0$. Hence, we obtain the following result.

\begin{prop}
Let $\mathbb{R}^7$ equipped with the pseudo-Riemannian metric $g$ given in \eqref{eq:metricexample}. Then, an irreducible isotropic spinor parallel with respect to $\nabla^{g,A}$ exists if $A$ satisfies \eqref{eq:torsionexample}.
\end{prop}

\noindent
This way we obtain plenty of metrics admitting parallel isotropic spinors for appropriate choices of contorsion. A direct inspection of Equation \eqref{eq:torsionexample} reveals the following result.
\begin{prop}
Let $\mathbb{R}^7$ equipped with the pseudo-Riemannian metric $g$ given in \eqref{eq:metricexample} and let $\varepsilon \colon \mathbb{R}^7 \to \mathbb{R}^8$ be an isotropic and irreducible spinor parallel with respect to $\nabla^{g,A}$ and whose associated triple of one-forms is \eqref{eq:tripleexamples}. Then:
\begin{equation*}
A(\partial_{y_i} , \partial_{y_j}, \partial_{z}) = 0\, , \qquad A(\partial_{y_i} , \partial_{y_j}, \partial_{y_k}) = 0
\end{equation*}

\noindent
for every $i,j,k = 1,2,3$.
\end{prop}

\noindent
Hence, the contorsion of parallel spinors on $(\mathbb{R}^7,g)$ with an associated triple of one-forms \eqref{eq:tripleexamples} is restricted and in particular, there are plenty of contorsion tensors for which no such parallel spinor exists. For the Levi-Civita connection, we immediately obtain the following corollary.

\begin{cor}\label{cor:LC-parallel_isotropic_spinor}
Let $\mathbb{R}^7$ equipped with the pseudo-Riemannian metric $g$ given in \eqref{eq:metricexample} and such that $K_i$, $i=1,2,3$, and $F$ depend only on the coordinates $(x_1,x_2,x_3)$. Then $(\mathbb{R}^7,g)$ admits an irreducible isotropic spinor parallel with respect to the Levi-Civita connection. 
\end{cor}

\begin{remark}
Note that in the previous corollary the functions $H_i$, $i=1,2,3$, are allowed to be arbitrary.
\end{remark}

\noindent
In the set-up of Corollary \ref{cor:LC-parallel_isotropic_spinor} the metric $g$ is generically not flat, not Ricci-flat, and not scalar-flat. In fact, a computation reveals that it is scalar-flat if and only if:
\begin{equation*}
e^{-2K_1} \frac{\partial^2H_1}{\partial y_1^2}+ e^{-2K_2} \frac{\partial^2H_2}{\partial y_2^2}+ e^{-2K_3} \frac{\partial^2H_3}{\partial y_3^2}=0.
\end{equation*}

\noindent
To the best of our knowledge, these are the first explicit examples of metrics admitting a torsion parallel irreducible isotropic spinor.

% % % % % % % % % % % % % % % % % % % % % % % % % % % % % % % % % % % % % % 
% % % % % % % % % % % % % % % % % % % % % % % % % % % % % % % % % % % % % %

\subsection*{Author contributions}

All authors have contributed equally.

\subsection*{Funding}

The work of AGG is supported by the Beijing Institute of Mathematical Sciences and Applications (BIMSA) and partially funded by the German Science Foundation (DFG) under Germany’s Excellence Strategy -- EXC 2121 “Quantum Universe” -- 390833306 and the Deutsche Forschungsgemeinschaft (DFG, German Research Foundation) -- SFB-Geschäftszeichen 1624 -- Projektnummer 506632645. The work of CSS was supported by the Leonardo grant LEO22-2-2155 of the BBVA Foundation.

\subsection*{Data availability}

No datasets were generated or analysed during the current study.

\subsection*{Conflict of interest}

Not applicable.

\subsection*{Ethics approval and consent to participate}

Not applicable.

% % % % % % % % % % % % % % % % % % % % % % % % % % % % % % % % % % % % % % 
% % % % % % % % % % % % % % % % % % % % % % % % % % % % % % % % % % % % % %

\bibliographystyle{amsplain}
\bibliography{biblio}

\end{document}